\documentclass{article}
\usepackage[utf8]{inputenc}
\usepackage{authblk}

\usepackage{latexsym,amsthm,amscd}
\usepackage{amsmath}
\usepackage{amssymb}
\usepackage{amsfonts}
\usepackage{tikz}
\usepackage{stmaryrd}
\usepackage{hyperref}
\usepackage[most]{tcolorbox}
\usetikzlibrary{patterns}
\usepackage{enumerate}
\usepackage{units}
\usepackage{CJKutf8}
\usepackage{ulem}

\title{Short-range and long-range order: a transition in block-gluing behavior in Hom shifts.}
\author[1]{Silvère Gangloff\footnote{Corresponding author}}
\affil[1]{
	AGH University of Krakow, Faculty of Applied Mathematics,
	al. Mickiewicza 30,
	30-059 Krak\'ow,
	Poland
	\url{sgangloff@agh.edu.pl}
	}

\author[2]{Benjamin Hellouin de Menibus}
\affil[2]{
Université Paris-Saclay, CNRS, Laboratoire Interdisciplinaire des Sciences du Numérique, 91400, Orsay, France
\url{hellouin@lisn.fr}
Orcid ID : 0000-0001-5194-929X
}

\author[1,3]{Piotr Oprocha}
\affil[3]{
National Supercomputing Centre IT4Innovations, University of Ostrava,
	IRAFM,
	30. dubna 22, 70103 Ostrava,
	Czech Republic
 \url{oprocha@agh.edu.pl}
	}

\date{\today}

\newtheorem{theorem}{Theorem}[section]
\newtheorem*{theorem*}{Theorem}
\newtheorem{lemma}[theorem]{Lemma}
\newtheorem{corollary}[theorem]{Corollary}
\newtheorem{question}[theorem]{Question}
\newtheorem{definition}[theorem]{Definition}
\newtheorem{example}[theorem]{Example}
\newtheorem{conjecture}[theorem]{Conjecture}
\newtheorem*{problem*}{Problem}
\newtheorem{notation}[theorem]{Notation}
\newtheorem{proposition}[theorem]{Proposition}
\newtheorem{remark}[theorem]{Remark}

\newtheorem{denomination}[theorem]{Denomination}

\tikzset{every loop/.style={min distance=2cm}}

\newcommand\bt[1]{\varphi(#1)}
\newcommand\quadcover[1]{\mathcal{U}^\square_{#1}}

\begin{document}

\maketitle

\begin{abstract}
\emph{Hom shifts}
form a class of multidimensional shifts of finite type (SFT) and consist of colorings of the grid $\mathbb Z^2$ where adjacent colours must be neighbors in a fixed finite undirected simple graph $G$. This class includes several
important statistical physics models such as the hard square model. The \emph{gluing gap} measures how far any two square patterns of size $n$ can be glued, and it can be seen as a measure of the range of order which affects the possibility to compute the entropy (or free energy per site) of a shift.
This motivates a study of the possible behaviors of the gluing gap. 
The class of Hom shifts is interesting because it allows us to describe dynamical properties, in particular mixing-type ones in the context of this article, using algebraic graph theory, which has received a lot of attention recently. Improving some former work of N.~Chandgotia and B.~Marcus, we prove that the gluing gap either depends linearly on $n$ or is dominated by $\log(n)$. We also find a Hom shift with gap $\Theta(\log(n))$, infirming a conjecture formulated by R.~Pavlov and M.~Schraudner. The physical interest 
of these results is to better understand the
transition from short-range to long-range order (respectively sublogarithmic and linear gluing gap), which is reflected in whether some parameter, the square cover, is finite or infinite.
\end{abstract}


\section{Introduction}

\textbf{Multidimensional shifts of finite type} are multidimensional dynamical systems defined as the action of the group $\mathbb Z^d$, via the shift, on 
the compact subset of $\mathcal{A}^{\mathbb{Z}^d}$ whose elements are the ones in which no pattern in $\mathcal{F}$ appear, where $\mathcal{A}$ is a finite set, $d \ge 2$, and 
$\mathcal{F}$ is a finite set of patterns on $\mathcal{A}$. They appear in various areas of mathematics: in particular as a straightforward generalisation of (unidimensional) subshifts of finite type, which themselves were first used by J.Hadamard in his work on geodesic flows on surfaces of negative curvature~\cite{H98}; in statistical physics, as lattice models such as the hard square model and square ice model; in mathematical logic, with the work of H.Wang~\cite{W61} on tilings of the plane with square tiles.

\paragraph*{Topological entropy.} The topological entropy of a shift of finite type, which in statistical physics is usually called free energy per site, is the asymptotic growth rate of the number of restrictions 
on $\llbracket 0,n-1\rrbracket^d$ of its elements. In a celebrated article, E.H.Lieb~\cite{L67} computed an exact formula of topological entropy for the square ice model, with the rationale of developing tools for computing efficiently some physically relevant quantities for models with high number of variables. Unfortunately, the method proposed by Lieb does not generalize to other shifts of finite type easily. In general, computing exactly topological entropy of a multidimensional shift of finite type is a very hard problem.

\paragraph*{Uncomputability of entropy in general.}
As a matter of fact, L.Hurd, J.Kari and K.Culik~\cite{HKC92} have proved that topological entropy is uncomputable for cellular automata, which form a subclass of multidimensional shifts of finite type. This leaves no hope to find a general method to compute topological entropy. Later, M.~Hochman and T.~Meyerovitch~\cite{HM11} 
provided a characterization of possible values of topological entropy on the class of shifts of finite type of dimension $d$, for any $d \ge 2$, as the non-negative real number which are computable from above, tightening the relation between multidimensional symbolic dynamics and computability theory. 

\paragraph*{Computability under block gluing property.}
On the other hand, algorithms have been developed in order to find rational approximations of 
topological entropy with arbitrary precision, in particular cases such as the hard square model~\cite{P10}. Furthermore, there exists a general method to compute topological entropy this way for shifts of finite type 
which satisfy the block gluing property in two dimensions $(d=2)$~\cite{PS15}. This property consists in the possibility to `glue' any pair of square patterns of the same size, provided that the distance between them is greater than a fixed constant, and then fill the grid $\mathbb{Z}^d$ into an element of the shift. 

\paragraph*{Edge of uncomputability.}
Together with M.~Sablik~\cite{GS18}, the first author proposed a quantification of the block gluing property, in which a function of the size of the patterns, called \textit{gap function}, replaces the constant. This function reflects the `range of order' in the system: the larger this function, the farther the presence of one particular pattern has 
an influence over the configuration. 
They studied the `edge of uncomputability' (by analogy with the edge of chaos) for this quantification, with the purpose of understanding \textit{how} the uncomputability phenomenon appears. 
In particular, they identified the `area' in between 
logarithmic and linear functions as where uncomputability occurs for topological entropy. 
Unfortunately, no tool is available in order to analyze shifts of finite type in this area. In particular, it is not known if there exists a shift of finite type in two 
dimensions which has the block gluing property with a gap function strictly between logarithmic and linear. 
On the class of decidable shifts, close to the one of shifts of finite type, the first and second authors
identified~\cite{GH18} a threshold at which uncomputability of 
entropy occurs, defined by the condition
\[
\sum_n \frac{f(n)}{n^2} = +\infty,
\]
where $f$ is the gap function. This suggests that if it is possible to find a similar threshold for shifts of finite type, it should be strictly in between logarithmic and linear.

\paragraph*{Hom shifts.} In order to understand better block 
gluing classes, which group together shifts of finite type having equivalent gap functions,
the strategy that we propose here is to restrict the scope to a natural subclass 
of the one of two-dimensional shifts of finite type. In this text, we consider Hom shifts, 
that is, the set of graph morphisms from $\mathbb{Z}^d$ to $G$, where $G$ is an undirected, simple, connected graph (self-loops are allowed). In the symbolic dynamics context, they were studied by N.~Chandgotia~\cite{C17}, who coined the name `Hom shifts', borrowing the definition from G.~Brightwell and P.~Winkler~\cite{BW00}. Although entropy is computable on Hom shifts \cite{F97}, this is a natural class to better understand the block-gluing property 
for the following reasons. \textbf{1.} First, it is related to statistical physics models, in the sense that the hard square shift is 
included in it, and the square ice model is related to the set of three-colourings of $\mathbb Z^2$ (in particular they have the same entropy), also in this class. \textbf{2.} While several problems are undecidable for multidimensional shifts of finite type, 
many of them become decidable for Hom shifts. For instance, 
Hom shifts are defined by a symmetric set of forbidden patterns, and such shifts have algorithmically computable entropy~\cite{F97} - a fact related to the intuition that it is impossible to embed universal computation in shifts under this (strong) constraint. 
\textbf{3.} The conceptual richness of graph theory should help to forge concepts in order to analyze 
block gluing classes in this restricted context, concepts 
which may then be extended to the general context of multidimensional shifts of finite type. 

\paragraph*{Mixing-type properties and algebraic topology.}
Mixing-type properties have been studied for Hom shifts 
in the recent years, in particular topological and measure-theoretical mixing \cite{B17, CM18, B21}.
In \cite{CM18}, the authors express gap functions for mixing-type properties of Hom shifts in terms of the diameter of 
the graph of walks of length $n$ on $G$. They use concepts of algebraic topological nature defined on finite graphs (as done in algebraic graph theory), in particular the universal cover - related to the fundamental group - in order to prove that whenever the graph $G$ is square-free, the gap function is $O(1)$ or $\Theta(n)$.
R.~Pavlov \& M.~Schraudner conjectured that this holds for general graphs as well (section \textbf{6.3} in \emph{op. cit.}). 

The concepts and tools used in~\cite{CM18} and the square free-hypothesis appear also in works related to homomorphism reconfiguration in graph theory, using a non-standard reconfiguration step (see for instance \cite{W20}). As well concepts of topological algebraic nature appear also in other works on symbolic dynamics, such as for instance the projective fundamental group~\cite{GP95}. For more details on the link between our tools and algebraic topology, see the second half of Section~\ref{section.universal}.

\paragraph*{This article.} We focus here on two-dimensional Hom shifts. We consider the problem of characterizing the possible equivalence classes of gap functions for a property 
slightly more general than classical block gluing, called `phased block gluing' by N.Chandgotia. 
Any Hom shift given by a finite graph $G$ is phased block gluing for some gap function, 
denoted by $\gamma_G$. The problem that we are interested in here is the following: 

\begin{problem*}
\textbf{1.} What are the possible equivalence classes $\Theta(\gamma_G)$ for all finite graphs $G$? \textbf{2.} Given a graph $G$, is it decidable which equivalence class the function $\gamma_G$ belongs to?\end{problem*}

Our main results, which address the first part of this problem, are the following:
\begin{theorem*}
For any finite graph $G$, $\gamma_G(n) = \Theta(n)$ or $O(\log(n))$ (Theorem~\ref{thm:main1}). There exists a graph $K$ such that $\gamma_K(n) = \Theta(\log(n))$ (Theorem~\ref{theorem.kenkatabami}).
\end{theorem*}

In particular, we disconfirm R.~Pavlov and M.~Schraudner's conjecture. In order to prove this theorem, we extend 
the methods developed in \cite{CM18}, and remove the 
square free hypothesis by considering, instead of the universal cover, its quotient by squares of $G$. 

The restriction to dimension two is due to the fact that some of our techniques cannot be easily generalized to higher dimensions. This applies in particular to the representation of cycles as `trees of simple cycles' (see Section~\ref{section.cactus}), and its further applications in Section~\ref{section.finite.square.cover}. Despite this restriction, this setting still covers a plethora of important examples, since several statistical physics models are two-dimensional.

Mathematically, the impact of this result is twofold. First, we deepen our understanding of the relationship between dynamical properties of Hom shifts (and by extension multidimensional shifts of finite type) and algebraic topology, because we prove a tight correspondence between behaviors of phased block gluing gap functions in the `upper part' of the spectrum of possible behaviors properties of the universal cover - finiteness or infiniteness of the quotient by squares - on the whole class of Hom shifts. Second, we develop technical tools which enable us to prove that there are no Hom shifts whose gluing gap function is intermediate in the sense that they are strictly between $O(\log(n))$ and $\Theta(n)$, which can serve as prototype for the general context of block-gluing in multidimensional shifts of finite type. The perspective of a complete classification of phased block gluing classes, motivated by a better understanding of these classes, allows us to expect further tools of topological algebraic nature to be developed along the way, which can be of interest in themselves, or may be useful in order to answer questions about Hom shifts or multidimensional shifts of finite type in general. 
In particular we hope that algebraic properties may help determining the edge of uncomputability, for computing or finding a closed form for entropy, but also other questions related to entropy, entropy minimality, and mixing-type properties. In this direction, we conjecture the following: 

\begin{conjecture}
Every two-dimensional Hom shift is either $O(1)$-phased block gluing, $\Theta(\log(n))$-phased block gluing or $\Theta(n)$-phased block gluing.
\end{conjecture}

\paragraph*{Physical interpretation.} 
We mentioned earlier that Hom shifts often appear as simple models in statistical physics. The above conjecture states that mixing properties of Hom shifts are \emph{rigid}, in the sense that they can be classified into three classes with no possible intermediate behavior. In this sense they do not correspond to phase transitions in the classical meaning of the term: a sudden change of behavior when some parameter (usually a real number) passes a threshold. Still, we believe that this phenomenon is related to phase transitions: if one considers a system of a family of systems that can be represented by Hom shifts and change some parameters, then this system must go from a mixing class to the next without any intermediate behavior. Since mixing properties are a description of the range of the order present in the system, this corresponds to a sudden change of behavior of the system that we hope to be mathematically tractable. In Section~\ref{section.conclusion}, we leave some open questions in this direction.

\paragraph*{Structure of the article}
First, in Section~\ref{sec:properties}, we relate properties of $G$ and the associated Hom shift $X_G$, and in particular how block-gluing on $X_G$ translates in terms of graph properties. In Section~\ref{section.decomposability}, we define the notion of universal cover, already used by Chandgotia. We introduce a notion of square decomposition for cycles of $G$, which lets us define a square cover $\quadcover{G}$ by quotienting the universal cover by the squares of $G$. This lets us prove that if the square cover of $G$ is infinite, then $\gamma_G (n) = \Theta(n)$. In Section~\ref{section.cactus}, we define a representation of cycles on $G$ as a tree of simple cycles. This representation is used in Section~\ref{section.finite.square.cover} to prove that if the square cover of $G$ is finite, then  $\gamma_G (n) = O(\log(n))$. In Section~\ref{section.log.existence}, we exhibit a graph $K$ such that $\gamma_K \in \Theta(\log)$. Finally, we briefly discuss in Section~\ref{section.conclusion} some problems that are left open.

\section{Definitions and notations\label{section.definitions}}

For any set $S$, we denote by $S^{*}$ the set of finite words on $S$. For a word $u$, we denote the number of its letters by $|u|$. We usually write $u$ as $u_0 \hdots u_{|u|-1}$. The empty word is denoted by $\epsilon$. Let us denote $\mathbb{N}^{*} = \{1, 2, \dots \}$ the set of positive integers and set
$\mathbb{N}=\mathbb{N}^{*} \cup \{0\}$. For all integers $a,b \in \mathbb{Z}$, we denote by $\llbracket a , b \rrbracket$ the interval $\{j \in \mathbb{Z} : a \le j \le b \}$. Similarly $\rrbracket a , b \llbracket:=\llbracket a+1 , b-1 \rrbracket$. Let also $\|\cdot\|_{\infty}$ be the norm defined by $\|\textbf{k}\|_{\infty} := \max(|\textbf{k}_1|,|\textbf{k}_2|)$ for all $k\in\mathbb Z^2$.

\subsection{Shifts}

Let us consider some finite set $\mathcal{A}$. A (two-dimensional) \textit{pattern} $p$ on $\mathcal{A}$ is an element of $\mathcal{A}^{\mathbb{U}}$, where $\mathbb{U}$ is a finite subset of $\mathbb{Z}^2$, and is called the \textit{support} of $p$. We say that $p$ \textit{appears} in an element $x$ of $\mathcal{A}^{\mathbb{Z}^2}$ when the restriction of $x$ to some $\textbf{u} + \mathbb{U}$, $\textbf{u} \in \mathbb{Z}^2$, is equal to $p$. 
A \textit{block pattern} is a pattern on support $\llbracket 0, n-1\rrbracket^2$ for some $n \ge 1$, which is called the \textit{size} of this pattern. 
The \textit{shift action} on $\mathcal{A}^{\mathbb{Z}^2}$ is the 
action of the group $\mathbb{Z}^2$ on this set defined by 
$\sigma_{\textbf{v}}(x) = (x_{\textbf{u}+\textbf{v}})_{\textbf{u} \in \mathbb{Z}^2}$ for all $\textbf{v} \in \mathbb{Z}^2$ and $x \in \mathcal{A}^{\mathbb{Z}^2}$.
We endow $\mathcal{A}^{\mathbb{Z}^2}$ with the product topology, defined by discrete topology on $\mathcal{A}$. This makes $\mathcal{A}^{\mathbb{Z}^2}$ a compact metrizable space and $\sigma_{\textbf{v}}$ continuous for every $\textbf{v}$.

A \textit{shift} on alphabet $\mathcal{A}$ is 
any compact subset $X$ of $\mathcal{A}^{\mathbb{Z}^2}$ which is  invariant under the shift action. A pattern is said to be \textit{globally admissible} for a shift $X$ when it appears in at least one of its elements.
Provided a set of patterns $\mathcal{F}$ on alphabet $\mathcal{A}$, we denote by $X_{\mathcal{F}}$ the shift on alphabet $\mathcal{A}$ whose elements are the ones in which no element of $\mathcal{F}$ appear. 
A shift $X$ is said to be of \textit{finite type} when there exists a finite $\mathcal{F}$ such 
that $X=X_{\mathcal{F}}$. Provided such a set $\mathcal{F}$, we say that a pattern $p$ is \textit{locally admissible} when no element of $\mathcal{F}$ appears in $p$. 

\subsection{Graphs}
In the whole text $G = (V_G,E_G)$ is some \textit{undirected}, simple and connected graph, where $V_G$ denotes the set of vertices of $G$ and $E_G$ its set of edges. Depending on the context, this graph may not necessarily be finite. Whenever we consider a graph $H$, we denote by $V_H$ the set of its vertices and $E_H$ the set of its edges. 

\begin{definition}
A \textbf{walk} on the graph $G$ is a non-empty word $p$ in $V_G^{*}$ such that for all $k \le |p|-2$, $(p_k, p_{k+1}) \in E_G$. We denote by $l(p)$ the number $|p|-1$, and call it the \textbf{length} of $p$ (equivalently, this is the number of edges that the walk follows). A \textbf{cycle} on $G$ is a walk $c$ 
such that $c_0 = c_{l(c)}$. It is said to be \textbf{simple} when $i < j$ and $c_i = c_j$ imply that $i= 0$ and $j = l(c)$. Similarly we say that a walk $p$ is  \textbf{simple} when for all $i \neq j$, $p_i \neq p_j$.
\end{definition}

\begin{notation}\label{notation.simple.cycle}
$\mathcal{C}^0_G$ is the set of simple cycles of $G$. 
\end{notation}

\begin{notation}\label{notation:opposite}
For all walk $p$, we will denote $p^{-1}$ the walk $p_{l(p)} \ldots p_0$.
\end{notation}

\begin{definition}
A \textbf{spine} on $a\in V_G$ is any cycle of length $2$ starting and finishing at $a$. A walk is said to be \textbf{non-backtracking} if it has no spine as a subword. 
\end{definition} 

\begin{notation}
We denote by $\varphi$ the function from the set of walks on $G$ to itself defined as follows: for every walk $p$, $\bt{p}$ is obtained from $p$ by replacing successively every spine $aba$ by $a$ until there is none left. $\varphi$ is well-defined because the remaining word does not depend on the order in which the spines are replaced. 
\end{notation}

\begin{notation}
For a walk $p$ of length $n \ge 1$ on $G$, say $p=p_0 \hdots p_{n}$, we denote by $\rho_l(p)$ (resp. $\rho_r(p)$) the set of walks of the form 
\[p_1 \hdots p_{n} x \quad (resp. \ x p_0 \hdots p_{n-1})\quad \text{for }x\in V_G.\] 
An element of this set is called a \textbf{left shift} (resp. a right shift) of $p$.
\end{notation}

\begin{notation}
For two walks $p$ and $q$ such that $q_0 = p_{l(p)}$, denote by $p \odot q$ the 
walk $p_0 \hdots p_{l(p)} q_1 \hdots q_{l(q)}$.
For any cycle $c$, denote by $c^n$, $n \ge 1$ the cycles such 
that for all $n \ge 2$, $c^n = c \odot c^{n-1}$ and $c^1 = c$.
\end{notation}

\begin{notation}
For any pair of vertices $a,b \in V_G$, we denote by $\delta(a,b)$ the shortest length of a walk in $G$ which begins at $a$ and ends at $b$. The \textbf{diameter} of $G$ is: 
\[\text{diam}(G) := \sup_{a,b \in V_G} \delta(a,b).\]
\end{notation}

\begin{definition}
A \emph{graph homomorphism} from $G_1$ to $G_2$ is a function $\psi: V_{G_1} \to V_{G_2}$ such that $(a,b)\in E_{G_1} \implies (\psi(a),\psi(b))\in E_{G_2}$.
\end{definition}

\subsection{Hom shifts}

\begin{notation}
The two-dimensional \textbf{Hom shift} corresponding to the graph $G$ is the shift $X_G$ on alphabet $V_G$ such that $x \in V_G^{\mathbb{Z}^2}$ is an element of $X_G$ if and only if 
for all $\textbf{u}, \textbf{v} \in \mathbb{Z}^2$ such that $\|\textbf{u}-\textbf{v}\|_{\infty}=1$, $x_{\textbf{u}}$ and $x_{\textbf{v}}$ are neighbors in $G$. 
\end{notation}

\begin{remark}
We may view $\mathbb Z^2$ as a graph such that $(u,v)\in E_{\mathbb Z^2}$ if and only if $\|u-v\|=1$. This way, each $x\in X_G$ can be viewed as a graph homomorphism $x\colon \mathbb Z^2\to G$. Then the set $X_G$ can be seen as the set of graph homomorphisms, which explains the name `Hom shift'.
\end{remark}

\begin{remark}
If $G$ is finite, then $X_G$ is a shift of finite type. Indeed, 
denoting by $\mathcal{F}_G$ the set of patterns $vw$ and $ \def\arraystretch{0.5}\begin{array}{c} v \\ w \end{array}$, where $(v,w) \notin E_G$, we have 
$X_G = X_{\mathcal{F}_G}$. Whenever we consider locally admissible patterns for $X_G$, this notion is relative to this set 
$\mathcal{F}_G$.  
\end{remark}

\begin{remark}
A pattern $p$ 
on support $\mathbb{U}$ is locally 
admissible for $X_G$ when there is a graph homomorphism from
$\mathbb{U}$ to $G$, where $\mathbb{U}$ is seen as a subgraph of the grid $\mathbb{Z}^2$.
\end{remark}

We denote by $L_G^n$ the set of walks of length $n$ on $G$.

\begin{notation}
For any integer $n \ge 0$, denote $X^{(n)}_G$ the subset of $\left(L^{n}_G \right)^{\mathbb{Z}}$ whose elements $x$ are such that there exists some $z \in X_G$ such that $x = z_{|\llbracket 0 , n \rrbracket \times \mathbb{Z}}$.
\end{notation}

\begin{notation}
For all $n \ge 1$, let us denote by $\Delta_G^n$ the graph whose vertices are 
the elements of $L^n_G$, and whose edges are the 
pairs $(p,q) \in L^n_G \times L^n_G$ such that $x_0 = p$ and $x_1 = q$ for some $x\in X^{(n)}_G$. For all $n$ and $(p,q) \in L^n_G \times L^n_G$, we denote by $d_G(p,q)$ the distance between $p$ and $q$, defined as the smallest length of a walk on $\Delta_G^n$ 
which begins at $p$ and ends on $q$. 
\end{notation}

In other words, the graph $\Delta_G^n$ tells which walks can be written next to each other, as vertical or horizontal patterns, in an element of $X_G$. 

\begin{remark}\label{remark.shift.neighbors}
For every walk $p$ and right or left shift $q$ of $p$ with $p\neq q$, we have $d_G(p,q) = 1$.
\end{remark}

\subsection{Block-gluing}

For two subsets $\mathbb{U}$ and $\mathbb{U}'$ of 
$\mathbb{Z}^2$, we set $\delta(\mathbb{U},\mathbb{U}') := \min_{\substack{\textbf{u} \in \mathbb{U}\\ \textbf{u}' \in \mathbb{U}'}} \|\textbf{u}-\textbf{u}'\|_{\infty}$.

\begin{definition}
Let us consider a function $f : \mathbb{N}^{*} \rightarrow \mathbb{N}^{}$, and an integer $k \in \mathbb{N}^{*}$. A shift on some alphabet $\mathcal{A}$ is said to be $(f,k)$-\textbf{phased block gluing} when, for every globally admissible block patterns $p$ and $p'$ having the same size $n$, 
and $\textbf{u},\textbf{u}' \in \mathbb{Z}^2$ such that 
\[\delta\left(\textbf{u} + \llbracket 0,n-1\rrbracket^2, \textbf{u}' + \llbracket 0,n-1\rrbracket^2 \right) \ge f(n),\]
there exists some $x\in X$ and some $\textbf{v} \in \mathbb{Z}^2$
such that $\|\textbf{v}\|_{\infty} < k$, $x_{\textbf{u} + \llbracket 0,n-1\rrbracket^2}=p$ and $x_{\textbf{u}'+\textbf{v} + \llbracket 0,n-1\rrbracket^2}=p'$. A shift which is $(f,1)$-\textbf{phased block gluing} for some $f$ is simply said to be $f$-\textbf{block gluing}. A shift which is $(f,k)$-phased block gluing for some $f$ and $k \ge 1$ is said to be \textbf{phased block gluing}.
\end{definition}

\begin{definition}
Whenever a shift $X$ is phased block gluing, we call phase of $X$ the minimal integer k such that $X$ is $(f, k)$-block gluing for some function $f$. We  denote the phase of $X$ by $\theta_X$.
\end{definition}

\begin{definition}
When a shift $X$ is phased-block-gluing, we denote by $\gamma_X : \mathbb{N}^\ast \rightarrow \mathbb{N}$ the minimal function such that $X$ is $(\gamma_X, \theta_X)$-phased block gluing. That is, for any function $f : \mathbb{N}^{*} \rightarrow \mathbb{N}$ such that $f(n)< \gamma_X(n)$ for some $n\in\mathbb{N}^\ast$, $X$ is not $(f,\theta_X)$-phased block gluing.
The function $\gamma_X$ is called the \textbf{gap function} of $X$ for the phased block-gluing property.
\end{definition}

In general it is difficult to compute exactly or obtain a concrete description of a gap function. We instead look at equivalence classes defined as follows:

\begin{notation}
Let us consider two functions $f,g : \mathbb{N}^{*} \rightarrow \mathbb{N}$. We write $g(n) = O(f(n))$
when there exist 
$c>0$ and $K>0$ such that for all $n$, 
\[g(n) \le c f(n) + K.\]
$O(g)$ is the set of functions $f$ such that $f(n) = O(g(n))$. 

We write $f(n) = \Theta(g(n))$ when $f(n) = O(g(n))$ and $g(n) = O(f(n))$. This defines an equivalence relation, and we denote by $\Theta(g)$ the equivalence class of $g$.
\end{notation}

\begin{denomination}
A shift $X$ is said to be $(\Theta(g),k)$-block gluing (resp. $(O(g),k)$-block gluing) when it is $(f,k)$-block gluing with $f \in \Theta(g)$ (resp. $O(g)$).
\end{denomination}




\section{Block-gluing of $X_G$ and properties of $G$}\label{sec:properties}

In this section, we analyze which properties of $G$ correspond to block-gluing on $X_G$. Note that if $G$ is not connected, $X_G$ cannot be $(f,k)$-phased block gluing for any $(f,k)$. 
Therefore, for the remainder of the text we assume that $G$ is connected.

\subsection{Distance between walks}

In this section, we prove that for any finite graph $G$, $X_G$ is phased block gluing and its phase is $1$ or $2$ (Proposition~\ref{prop:block-gluing}). 
For simplicity, we use the notation $\gamma_G := \gamma_{X_G}$.

A subset $\mathbb{U} \subset \mathbb{Z}^2$ is said to be \textit{connected} when the corresponding subgraph of $\mathbb{Z}^2$ is connected.

\begin{definition}
A finite set $\mathbb{U} \subset \mathbb Z^2$ is said to be \emph{block-like} when it is connected and for every $k \in \mathbb{Z}$, $\mathbb{U} \cap (\{k\}\times \mathbb Z)$ and $\mathbb{U} \cap (\mathbb Z\times \{k\})$ are intervals.
\end{definition}

When the support is a rectangle, that is, the product of two intervals, the following lemma corresponds to Proposition~\textbf{2.1} in~\cite{CM18}.

\begin{lemma}\label{lem:localglobal}
Every pattern $p$ which is locally admissible for $X_G$ and whose support is block-like is globally admissible.
\end{lemma}

\begin{proof}
Let us fix a locally admissible pattern $p$ on a block-like support $\mathbb{U}$. We define some configuration $x$ in which $p$ appears. We first set $x_{|\mathbb{U}}=p$.

Because $\mathbb{U}$ is block-like, a point $(i,j) \notin \mathbb{U}$ cannot have more than two neighbors in $\mathbb{U}$. Let us assume that there is some $(i,j) \notin \mathbb{U}$ which has two neighbors in $\mathbb{U}$. 
These neighbors are respectively 
of the form $(i \pm 1, j)$ and $(i,j \pm 1)$,  otherwise 
$\mathbb{U} \cap (\{j\}\times \mathbb Z)$ or $\mathbb{U} \cap (\mathbb Z\times \{j\})$ would not be an interval. Without loss of generality, let us assume that they are $(i + 1, j)$ and $(i,j + 1)$ (the other cases are dealt with similarly). Let us prove that $(i+1,j+1) \in \mathbb{U}$.
Since $\mathbb{U}$ 
is connected, there exists a walk on $\mathbb{U}$ from $(i + 1, j)$ to $(i,j + 1)$. 
Such a walk intersects $\{i\}\times \rrbracket 
{-\infty}, j\llbracket$ or $\{i+1\}\times \rrbracket j, +\infty\llbracket$ (see for instance Figure~\ref{figure.intersection}).

\begin{figure}
\begin{center}
    \begin{tikzpicture}[scale=0.3]
    \fill[pattern=north west lines] (0,0) rectangle (1,1);
    \fill[color=gray!60] (0,1) rectangle (1,2);
    \fill[color=gray!60] (1,0) rectangle (2,1);
    \draw (0,0) grid (2,2);
    \draw[fill=gray!25] (2,0) rectangle (3,1);
    \draw[fill=gray!25] (3,0) rectangle (4,1);
    \draw[fill=gray!25] (3,1) rectangle (4,2);
    \draw[fill=gray!25] (3,2) rectangle (4,3);
    \draw[fill=gray!25] (3,3) rectangle (4,4);
    \draw[fill=gray!25] (2,3) rectangle (3,4);
    \draw[fill=gray!25] (1,3) rectangle (2,4);
    \draw[fill=gray!25] (0,3) rectangle (1,4);
    \draw[fill=gray!25] (0,2) rectangle (1,3);

    \begin{scope}[xshift=-10cm]
    \fill[pattern=north west lines] (0,0) rectangle (1,1);
    \fill[color=gray!60] (0,1) rectangle (1,2);
    \fill[color=gray!60] (1,0) rectangle (2,1);
    \draw (0,0) grid (2,2);
    \draw[fill=gray!25] (2,0) rectangle (1,-1);
    \draw[fill=gray!25] (1,-1) rectangle (2,-2);
    \draw[fill=gray!25] (0,-2) rectangle (1,-1);
    \draw[fill=gray!25] (-1,-2) rectangle (0,-1);
    \draw[fill=gray!25] (-2,-2) rectangle (-1,-1);
    \draw[fill=gray!25] (-2,-1) rectangle (-1,0);
    \draw[fill=gray!25] (-2,0) rectangle (-1,1);
    \draw[fill=gray!25] (-2,1) rectangle (-1,2);
    \draw[fill=gray!25] (-1,1) rectangle (0,2);
    
    \end{scope}

    \begin{scope}[xshift = 10cm]
        \draw (0,3) rectangle (1,4);
        \node[scale=0.7] at (4,3.5) {$\boldsymbol{(i+1,j+1)}$};
        \filldraw[fill=gray!60] (0,1.5) rectangle (1,2.5);
        \node[scale=0.7] at (5.5,2) {$\boldsymbol{(i+1,j)}$ or $\boldsymbol{(i,j+1)}$};
        \filldraw[pattern=north west lines] (0,0) rectangle (1,1);
        \node[scale=0.7] at (2.5,.5) {$\boldsymbol{(i,j)}$};
        \filldraw[fill=gray!25] (0,-1.5) rectangle (1,-.5);
        \node[scale=0.7] at (5,-1) {other elements of $\mathbb U$};
    \end{scope}
    \end{tikzpicture}
\end{center}
\caption{Illustration of the proof of Lemma~\ref{lem:localglobal}: two possible paths from $(i+1,j)$ to $(i,j+1)$ in $\mathbb{U}$. Only the one on the right can be contained in $\mathbb{U}$.\label{figure.intersection}}
\end{figure}

Since $\mathbb{U}$ is block-like and $(i,j) 
\notin \mathbb{U}$, it can't intersect the first set and thus intersects the second. Again, since $\mathbb{U}$ is block-like, we have $(i+1,j+1) \in \mathbb{U}$.

Every block-like subset of $\mathbb{Z}^2$ such that no $(i,j) \notin \mathbb{U}$ has two neighbors in $\mathbb{U}$ is a rectangle. Indeed, for every $j \in \mathbb{Z}$ such that the columns $\mathbb{Z} \times \{j\}$ and $\mathbb{Z} \times \{j+1\}$ intersect $\mathbb{U}$ non-trivially, 
their respective intersections with $\mathbb{U}$ are equal. A similar statement is 
satisfied for rows. This implies that $\mathbb{U}$ is a rectangle. 

Let us consider the minimal rectangle $R$ which contains the set $\mathbb{U}$. There 
exists a sequence $(\textbf{v}_l)_{1 \le l \le m}$ of elements of $\mathbb{Z}^2$ such that, denoting
\[\forall l \in \llbracket 0,m\rrbracket,\  \mathbb{U}_l := \mathbb{U} \cup \{\textbf{v}_1 , \dots , \textbf{v}_l\},\] we have $\mathbb{U}_m = R$ and, for all $l <m$, $\textbf{v}_{l+1}$ is not in $\mathbb{U}_l$ but has exactly two neighbors in $\mathbb{U}_l$. 

\begin{figure}
\begin{center}
    \begin{tikzpicture}[scale = 0.5]
    \foreach \x/\y in {0/1, 0/2, 0/3, 1/0, 1/1, 1/2, 1/3, 1/4, 2/0, 2/1, 2/2, 2/3, 2/4, 3/0, 3/1, 3/2, 4/0}
    {
        \filldraw[fill = gray!25](\x, \y) rectangle ++(1,1);
    }

    \node[scale=0.8] at (3.625, 3.5) {$v_{l+1}$};
    \node[scale=0.8] at (2.5, 2.5) {$w_l$};
    \end{tikzpicture}
    \end{center}
    \caption{\label{fig:growrectangle}A possible choice for $v_{l+1}$ and the corresponding $w_l$.}
\end{figure}
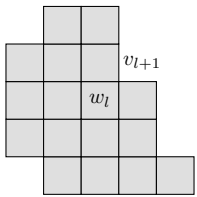

For all $l<m$, $\mathbb{U}_l$ is block-like, and therefore  there is a an element $\textbf{w}_l \in \mathbb{U}_l$ which is neighbor of the two neighbors of $\textbf{v}_{l+1}$ in $\mathbb{U}_l$ (see Figure~\ref{fig:growrectangle}). We set $x_{\textbf{v}_{l+1}} := x_{\textbf{w}_l}$. 

The defined pattern $x_{|R}$ is locally admissible on a support $R$, which is a rectangle, so it is globally admissible by Proposition~\textbf{2.1} in~\cite{CM18}. Hence $p$ is globally admissible. 
\end{proof}

The following characterization is well-known:

\begin{lemma}\label{lem:bipartite}
	A graph $H$ is bipartite if and only if it has no cycle of odd length. 
\end{lemma}

\begin{lemma}\label{lem:mixingbipartite}
Let $H$ be a finite graph. For every $u,v \in V_H$ and $k\geq \textrm{diam}(H)$, there is a walk from $u$ to $v$ of length $k$ or $k+1$. If $H$ is not bipartite, then for all $k\geq 3\textrm{diam}(H)$, there is a walk  from $u$ to $v$ of length $k$.
\end{lemma} 

\begin{proof}
By definition of the diameter, there is a walk $p$ from $u$ to $v$ whose length is at most $\textrm{diam}(H)$. For all spine $t$ on $v$, $p \odot t$ has length $l(p)+2$ and also begins at $u$ and ends on $v$. The first claim follows.

Let us assume that $H$ is not bipartite. 
This implies that $H$ contains a cycle $c$ of odd length (see Lemma~\ref{lem:bipartite}). Let $p$ and $q$ be the shortest walk from $u$ to $c_0$ and from $c_0$ to $v$, respectively. The respective lengths of walks $p \odot q$ and $p \odot c \odot q$ have different parities and are both smaller than $k$. Let us denote $r$ the one which has the same parity as $k$. Then for any spine $t$ on $v$, the cycle $r \odot t^{(k-l(r))/2}$ 
has length $k$ and is from $u$ to $v$. \end{proof}

\begin{lemma}
For all $n \ge 1$, the graph $\Delta_G^n$ is bipartite if and only if $G$ is bipartite. 
\end{lemma}

\begin{proof}
If $\Delta_G^n$ is not bipartite, it has a cycle of odd length. Along this cycle, if we take the last letter of the walk on $G$ corresponding to each vertex, we get a cycle in $G$ of odd length. 
Reciprocally, assume $G$ has a cycle $c$ of odd length $m$. For all $j$ between $0$ and $m-1$,
we denote by $p^{(j)}$ the walk on $G$ which begins at $c_j$ and alternates between $c_j$ 
and $c_{j+1}$. Then $p^{(0)} \ldots p^{(m-1)} p^{(0)}$ is a cycle of odd length on $\Delta_G^n$. 
\end{proof}

\begin{proposition}\label{prop:block-gluing}
Let $d_G (n) := \textrm{diam}(\Delta_G^n)$.
The shift $X_G$ is $(\Theta(d_G),2)$-phased-block-gluing. When $G$ is not bipartite, it is 
also $\Theta(d_G)$-block-gluing.
\end{proposition}

This implies that $\gamma_{X_G} \in \Theta(d_G)$, generalizing Proposition~\textbf{4.1} in \cite{CM18}. 

\begin{proof}
Let us first prove the second claim. We assume that $G$ is not bipartite. We show that $X_G$ is $O(d_G)$-block gluing. Consider some $k\geq 3d_G(n)$ and let $p$ and $q$ be two locally admissible block patterns of size $n$ and $\textbf{u},\textbf{v}$ such that 
\[\delta\left(\textbf{u}+\llbracket 0, n-1\rrbracket ^2,\textbf{v}+\llbracket 0, n-1\rrbracket ^2\right) = k.\]

By Lemma~\ref{lem:localglobal}, there exist $x, y \in X_G$ such that $x_{\textbf{u}+\llbracket 0, n-1\rrbracket^2} = p$ and  $y_{\textbf{v}+\llbracket 0, n-1\rrbracket^2} = q$. Without loss of generality, we can assume that $\textbf{v}_1 \ge \textbf{u}_1$ and $\textbf{v}_2 \ge \textbf{u}_2$, as well as $\textbf{v}_1 - \textbf{u}_1 = n + k-1$.

We apply Lemma~\ref{lem:mixingbipartite} on the graph $H=\Delta_G^n$ and obtain that there exists a walk of length exactly $k$ from $x_{(\textbf{u}_1,\textbf{v}_2)+\{n-1\}\times \llbracket 0,n-1 \rrbracket}$ to $y_{\textbf{v} + \{0\}\times \llbracket 0,n-1 \rrbracket}$. This corresponds to a locally admissible pattern $p'$ on support $\llbracket 0,k \rrbracket \times \llbracket 0,n-1 \rrbracket$.
Let us denote by $\mathbb{V}$ the following set: 
\[
\mathbb{V} = \left (\mathbf{u}+\llbracket 0, n-1\rrbracket \times \llbracket 0, n-1+ \textbf{v}_2 - \textbf{u}_2 \rrbracket
\right) \bigcup \left (\mathbf{v}+\llbracket \textbf{u}_1 - \textbf{v}_1, n-1\rrbracket \times \llbracket 0, n-1 \rrbracket
\right)\]

\begin{figure}[h!]
\begin{center}
\begin{tikzpicture}[scale=0.3]
\draw[dashed, fill=gray!40] (0,0) -- (0,5) -- (11,5) -- (11,2) -- (3,2) -- (3,0) 
-- (0,0);

\node at (-0.5,-0.5) {\textbf{u}};
\draw (0,0) rectangle (3,3);

\node at (7.5,1.5) {\textbf{v}};
\draw (8,2) rectangle (11,5);
\node at (1.5,1.5) {p};
\node at (9.5,3.5) {q};
\draw[latex-latex] (3,-1) -- (8,-1);
\draw[latex-latex] (12,2) -- (12,5);
\node at (13,3.5) {n};
\node at (5.5,-2) {k};

\node at (5.5,6) {$\mathbb{V}$};

\begin{scope}[xshift=17.5cm]
\draw[dashed, fill=gray!40] (0,0) -- (0,9) -- (11,9) -- (11,6) -- (3,6) -- (3,0) 
-- (0,0);

\node at (-0.5,-0.5) {\textbf{u}};
\draw (0,0) rectangle (3,3);

\begin{scope}[yshift=4cm]
\node at (7.5,1.5) {\textbf{v}};
\draw (8,2) rectangle (11,5);
\node at (9.5,3.5) {q};
\draw[latex-latex] (12,2) -- (12,5);
\node at (13,3.5) {n};
\node at (5.5,6) {$\mathbb{V}$};
\end{scope}

\node at (1.5,1.5) {p};
\draw[latex-latex] (3,-1) -- (8,-1);

\node at (5.5,-2) {k};

\end{scope}

\end{tikzpicture}
\end{center}
\caption{Illustration of the definition of $\mathbb{V}$ in two different situations.\label{figure.v}}
\end{figure}
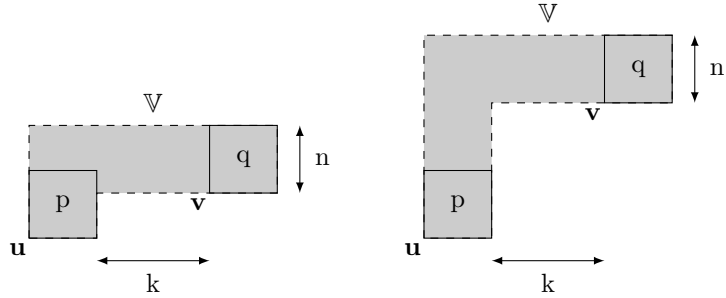

One can find an illustration of $\mathbb{V}$ 
on Figure~\ref{figure.v}. We define a locally admissible pattern $w$ on $\mathbb{V}$ 
by setting  
\[
w|_{(\textbf{u}_1+n-1,\textbf{v}_2) + (\llbracket 0,k \rrbracket \times \llbracket 0, n-1\rrbracket)} = p',
\]\[
 w|_{\textbf{v}+\llbracket 0, n-1\rrbracket^2} = q
\
\]
and such that $w$ coincides with $x$ on the remainder of $\mathbb{V}$.

By definition of $p'$, $w$ is well defined where the above three rectangles intersect.
Since $\mathbb{V}$ is a block-like set, by Lemma~\ref{lem:localglobal}, $w$ is globally admissible. This proves 
that $X_G$ is $O(d_G)$-block gluing, meaning that $\gamma_G = O(d_G)$. 
By the definition $\gamma_G (n) \geq d_G(n)$, so we obtain that $d_G = O(\gamma_G)$. 
 
In the case $G$ is bipartite, the proof follows the same lines, except that Lemma~\ref{lem:mixingbipartite} provides a walk of length $k$ or $k+1$, depending on parity.
In the latter case we have to shift $y|_\mathbf{v}$ by one column to the right (modifying $\mathbb{V}$ accordingly). We obtain this way that $X_G$ is $(O(d_G),2)$-phased block gluing. 
\end{proof}

\subsection{From walks to cycles}

First, we prove that we only need some values to determine the class of the function $\gamma_G$.

\begin{lemma}\label{lemma.even}
For all $n$, we have $\gamma_G(n) \le\gamma_G(n+1) \le\gamma_G(n)+2$. In particular, the equivalence class $\Theta(\gamma_G)$ is the same as the equivalence class of the function $n \mapsto \gamma_G(k\lfloor n/k \rfloor)$ for any $k>0$. 
\end{lemma}

\begin{proof}
The inequality $\gamma_G(n) \le\gamma_G(n+1)$ is trivial. Let us prove the second one. Let us consider two walks $p$ and $q$ of length $n+1$. 
There exists $m \le\gamma_G(n)$ and a walk $p^{(0)}, \dots , p^{(m)}$ in $\Delta^{n+1}_G$ from $p^{(0)} = p_0 \hdots p_{n}$ to $p^{(m)} = q_0 \hdots q_{n}$. For all $i \ge 1$, let us set $q^{(i)} = p^{(i)} p^{(i-1)}_n$ and $q^{(0)} = p$. For all $i < m$, $(q^{(i)},q^{(i+1)})$ is an 
edge of $\Delta^{n+1}_G$, and $q^{(0)} = p$. Since $q^{(m)}$ is equal to $q$ except for the last vertex, $q^{(m)} \in \rho_r (\rho_l(q))$. As a consequence $d_G(p,q) \le m+2$, and thus $\gamma_G(n+1) \le\gamma_G(n)+2$.
\end{proof}

In Proposition~\ref{prop:block-gluing}, we related $\gamma_G$ to the diameter of the graph. For bipartite graphs, it is enough to consider the distance between cycles.

\begin{lemma}\label{lemma.walks.to.cycles}
Let us assume that $G$ is bipartite. For every walk $p$ of even length, there exists a cycle $c$ such that $d_G(p, c)\leq \text{diam}(G)+1$.
\end{lemma}

\begin{proof}
Consider a walk $p$ of even length $n$. The result is clear when $n \leq\text{diam}(G)$, so we assume that $n >\text{diam}(G)$.
\begin{enumerate}
\item \textbf{There exists $m \le\text{diam}(G)+1$ such that there is a walk from $p_{n}$ to $p_m$ of length 
$m$ or $m-1$:} for all $k \in \{0, \dots, n\}$, let us set $l_k := k - \delta(p_n,p_k)$. We have that $l_0 = -\delta(p_n,p_0) \le 0$, and $l_{\text{diam}(G)+1} \ge 1$ by definition of the diameter. 
Furthermore, for all $k < n$, we have: 
\[ 0 \le l_{k+1} - l_k \le 2.\]
As a consequence there exists an integer $m \le\text{diam}(G)+1$ such that $l_m \in \{0,1\}$, which means that $\delta(p_n,p_m)$ is $m$ or $m-1$.

\item \textbf{We have $\delta(p_n,p_m)=m$:} let $q$ be a walk of length $\delta(p_n,p_m)$ from $p_n$ to $p_m$. If $q$ is of length $m-1$, the walk $p_m \hdots p_n \odot q_0 \hdots q_{m-1}$ is a cycle of odd length $n-1$, which is impossible since $G$ is bipartite; it follows that $q$ is of length $m$.
 
\item \textbf{Conclusion:} define, for all $k \le m$, the walk $p^{(k)} = p_k \hdots p_n \odot q_0 \hdots q_k$. 
Then $p^{(0)} = p$, $p^{(k)}$ and $p^{(k+1)}$ are neighbors in $\Delta^n_G$ for all $k \le m-1$, 
and $p^{(m)}$ is a cycle. \qedhere
\end{enumerate}
\end{proof}

\section{Decomposability of simple cycles into squares\label{section.decomposability}}

\subsection{Universal cover\label{section.universal}}

The notion of universal cover comes from the notion of topological covering space and can be defined in an abstract manner as a universal object with regards to so-called graph coverings. This point of view is well-explained in \cite{S83} Section 4.1 or \cite{CM18} Section 5. Below we present an explicit construction due to D.~Angluin \cite{A80}.

\begin{definition}
For every $a \in V_G$, let $\mathcal{U}_G[a]$ be the graph whose vertices are the non-backtracking walks on $G$ beginning at $a$ and whose edges are the pairs of walks $(p,q)$ such that $p = qv$ or $q = pv$ for some vertex $v \in V_G$. 
\end{definition}

\begin{notation}
For two vertices $a,b$ which are neighbors in $G$, let $\psi_{a \mapsto b} : \mathcal{U}_G[a] \rightarrow \mathcal{U}_G[b]$ be the graph morphism defined by 
$\psi_{a \mapsto b}(q) = \varphi(b q)$ for all $q \in V_{\mathcal{U}_G[a]}$. 
\end{notation}

\begin{lemma}\label{lemma.morphism}
For all $a,b$ neighbors in $G$, 
$\psi_{a \mapsto b} \circ \psi_{b \mapsto a} = id_{V_{\mathcal{U}_G[b]}}$.
\end{lemma}

\begin{proof}
Let $p$ be a non-backtracking walk beginning at $b$. Since $p_0 = b$ and spines can be removed in any order, we have that $\psi_{a \mapsto b} \circ \psi_{b \mapsto a} = \varphi(bap) = p$.
\end{proof} 

\begin{lemma}
All the graphs $\mathcal{U}_G[a]$, $a \in V_G$, are isomorphic. 
\end{lemma}

\begin{proof}
By Lemma~\ref{lemma.morphism}, $\mathcal{U}_G[a]$ and $\mathcal{U}_G[b]$
are isomorphic for all $a$ and $b$ neighbors in $G$. Since $G$ is connected, this is sufficient. 
\end{proof}

The corresponding isomorphic class is usually called the \textbf{universal cover} of $G$, denoted by $\mathcal{U}_G$, and is thought as an unlabeled graph which admits labelings having an interpretation in terms of walks on $G$.\bigskip

In the following sections, we introduce concepts which are built upon the universal cover, in particular what we call the square cover (Section~\ref{sec:sqcover}). Let us take a moment here to explain the intuition behind this construction. We would like to establish a correspondence between $X_G$ and  $X_{\mathcal{U}_G}$, in order to deduce the block gluing gap function of $X_G$ from the one of $X_{\mathcal{U}_G}$: the latter is much easier to deal with since $\mathcal{U}_G$ is a tree.

We proceed as follows. We associate to any $x\in X_G$ a configuration $\overline{x}$ in $X_{\mathcal{U}_G}$ in the following way, using labeling $\mathcal{U}_G[x_{\textbf{0}}]$ for $\mathcal{U}_G$: first set $\overline{x}_{\textbf{0}} = x_{\textbf{0}}$; then for all other $\textbf{i}\in\mathbb Z^2$, choose a walk $p$ from $\textbf{0}$ to $\textbf{i}$, and set $\overline{x}_\textbf{i}$ equal to $\varphi(x_{p_0}x_{p_1}\dots x_{p_{l(p)}})$ (in particular, $\eta(\overline{x}_i) = x_i$).
However this definition makes sense only if the value of $\overline{x}_i$ is independent from the choice of walk $p$. This is the case only if, for every $x\in X_G$ and every cycle in $\mathbb Z^2$ (see an illustration on Figure~\ref{fig:funda}), the corresponding cycle in $G$ backtracks to a trivial cycle.

\begin{figure}[h!]
\begin{center}
\begin{tikzpicture}[every node/.style={draw,circle, minimum size=.6cm, inner sep=0pt}, baseline={([yshift=-.5ex]current bounding box.center)}]
\node (a) at (0,0) {a};
\node (b) at (1.6,0) {b};
\node (c) at (.8,.8) {c};
\draw (a) -- (b) -- (c) -- (a);
\path (a) edge[loop left,looseness=5, in=150, out=210] (a);
\end{tikzpicture}
\hspace{2cm}
\begin{tikzpicture}[scale=0.6, baseline={([yshift=-.5ex]current bounding box.center)}]
\draw[dotted] (0,0) grid (5,3);
\foreach \x/\y/\z in {0/0/a, 0/1/b, 0/2/a, 1/0/c, 1/1/a, 1/2/b, 2/0/a, 2/1/b, 2/2/c, 3/0/a, 3/1/c, 3/2/a, 4/0/c, 4/1/b, 4/2/a }{
\node at (\x+.5, \y+.5) {$\z$};}
\draw[->] (.5,.5) -- (4.5,.5) -- (4.5, 2.5) -- (.5, 2.5) -- (.5, .5);
\end{tikzpicture}
\caption{On the \textbf{left}: a graph $G$ containing a cycle of length four ($aabca$). On the \textbf{right}: the cycle $w = acaacbaacbaba$ obtained by following 
a loop in $\mathbb Z^2$ in a configuration of $X_G$.
\label{fig:funda}}
\end{center}
\end{figure}
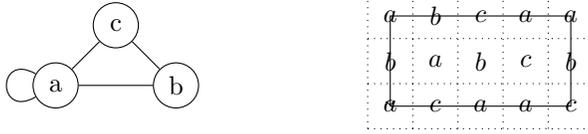

On the one hand, the set of cycles which are obtained as above, by following a cycle of $\mathbb Z^2$ in a configuration of $X_G$, corresponds to trivial cycles in the \emph{fundamental group} of the subshift $X_G$, a notion that was introduced in \cite{GP95} for general subshifts and received some recent attention \cite{PV22}. On the other hand, the universal cover of a graph is related to its own fundamental group: a cycle is trivial if and only if removing backtracks yields the empty cycle; removing backtracks plays the role of homotopy in this context. 

With this homotopic point of view, the definition above makes sense when the fundamental group of $X_G$ and the fundamental group of $G$ are isomorphic. The results of \cite{CM18} use implicitely the fact that this is true when $G$ does not contain any cycle of length $4$: this was their driving hypothesis. However this is not true in general.

In the general case, we can see that the structure of loops in $\mathbb Z^2$ implies that cycles in $X_G$ correspond to cycles in $G$ that can be, in some sense, decomposed into cycles of length $4$. Therefore the fundamental group of $X_G$ is related to the quotient of the universal cover by cycles of length $4$.
This is the intuition behind the definition of the square cover.

\subsection{Square decomposition}\label{sec:square}

\begin{denomination}
A \textbf{square} is a non-backtracking cycle of length four.
\end{denomination}

In the present section, we define the notion of decomposability into squares for simple cycles of $G$.

\begin{notation}
    For two cycles $c$ and $c'$ and $k \le l(c)$ such that $c_k = c'_0$, the cycle $c {\oplus}_{k} c'$ is defined by:
\[c {\oplus}_{k} c' = c_0 \hdots c_{k-1} c' c_{k+1} \hdots c_{l(c)}.\]
\end{notation}

\begin{definition}
Let us consider two non-backtracking walks $p,q$. We say that $p$ and $q$ \textbf{differ by a square} when there exists an integer $k$ and a square $s$ such that 
$q = \varphi(p {\oplus}_k s)$ or 
$p = \varphi(q {\oplus}_k s)$.
Figure~\ref{figure.differ} illustrates the types of pairs $(p,q)$ of walks which differ by a square. 

Let $\sim_{\square}$ be the transitive closure of this relation between walks.
\end{definition}

\begin{figure}[h!]

\begin{center}
\begin{tikzpicture}[scale=0.3]
\draw (0,0) -- (1,1) -- (2,1) -- (3,2) -- (4,1) -- (5,1) -- (6,2);
\draw[-latex] (0,0) -- (0.5,0.5);
\draw[-latex] (5,1) -- (5.75,1.75);
\draw[-latex] (2,1) -- (2.75,1.75);
\draw[-latex] (2,1) -- (2.75,0.25);

\draw (2,1) -- (3,0) -- (4,1);
\draw[dashed] (6,2) -- (7,3);
\draw[dashed] (0,0) -- (-1,-1);
\node at (3,-1) {p};
\node at (3,3) {q};
\node at (-2,1) {\textbf{(i)}};

\begin{scope}[xshift=12cm]
\draw (0.5,0) -- (1.5,1) -- (2.5,1) -- (4,1) -- (5,1) -- (6,2);
\draw[-latex] (0.5,0) -- (1,0.5);
\draw[-latex] (2.5,1) -- (3.5,1);
\draw[-latex] (2.5,-0.5) -- (3.5,-0.5);
\draw[-latex] (5,1) -- (5.75,1.75);
\draw (2.5,1) -- (2.5,-0.5) -- (4,-0.5) -- (4,1);
\draw[dashed] (6,2) -- (7,3);
\draw[dashed] (0.5,0) -- (-0.5,-1);
\node at (3.25,-1.5) {p};
\node at (3.25,2) {q};
\node at (-2,1) {\textbf{(ii)}};
\end{scope}

\begin{scope}[xshift=24cm]
\draw (0,0) -- (1,1) -- (2,1) -- (4,1) -- (5,1) -- (6,2);
\draw[-latex] (0,0) -- (0.5,0.5);
\draw[-latex] (5,1) -- (5.75,1.75);
\begin{scope}[yshift=-3cm]
\draw (3,1) -- (2,0) -- (3,-1) -- (4,0) -- (3,1);
\draw[-latex] (3,-1) -- (3.75,-0.25);
\end{scope}
\draw (3,1) -- (3,0.5);
\draw[dashed] (3,0.5) -- (3,-1.5);
\draw (3,-1.5) -- (3,-2);
\draw[dashed] (6,2) -- (7,3);
\draw[dashed] (0,0) -- (-1,-1);
\node at (3,-5) {p};
\node at (3,2) {q};
\node at (-2,1) {\textbf{(iii)}};
\end{scope}
\end{tikzpicture}
\end{center}
\caption{Partial representation of two walks $p,q$ which differ by a square.\label{figure.differ}}
\end{figure}
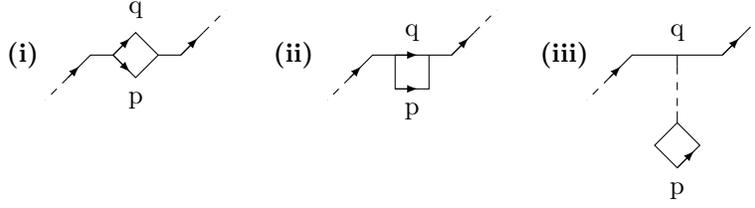

\begin{definition}\label{definition.decomposable}
A \textbf{square decomposition} of a simple cycle $c$ on $G$ is any sequence $(p^{(i)})_{0\le i \le m}$ of non-backtracking walks such that $c = p^{(0)}$, $p^{(m)}$ is an empty cycle, and $p^{(i)}$ and $p^{(i+1)}$ differ by a square for all $i<m$. The cycle $c$ is said to be \textbf{decomposable into squares} when such a decomposition exists, which is equivalent to $c\sim_{\square} c'$ for some empty cycle $c'$. The smallest length $m$ for which $c$ has a decomposition $(p^{(i)})_{0\le i \le m}$ is called the \textbf{area} of $c$ and is denoted $m_c$. 
\end{definition}

\begin{definition}
A graph is said to be \textbf{square-decomposable} when all of its simple cycles are decomposable into squares.
\end{definition}

It is straightforward that when $G$ is square-decomposable, every cycle (non necessarily simple) of $G$ is also decomposable into squares. It follows that:

\begin{lemma}\label{lemma.even.length}
A square-decomposable graph is bipartite.
\end{lemma}

\begin{proof}
It is sufficient to see that for $p,q$ two walks which differ by a square, $l(p)-l(q)$ is even. Therefore every cycle of a square-decomposable graph is of even length.
\end{proof}

As a consequence of Lemma~\ref{lemma.even.length} and Lemma~\ref{lemma.walks.to.cycles} we have the following: 

\begin{corollary}
If $G$ is square-decomposable, then $\Theta(\gamma_G)$ is the class of the function $n\mapsto\max d_G(c,c')$ where the maximum is taken on all cycles $c$ and $c'$ of length $n$.
\end{corollary}

\subsection{Square cover\label{sec:sqcover}}

In this section, we define the square cover, which is central in the dichotomy between $\Theta(n)$-block gluing and $O(\log(n))$-block gluing Hom shifts.

\begin{notation}
Denote by $\quadcover{G}[a]$ the quotient of $\mathcal{U}_G[a]$ by the equivalence relation $\sim_{\square}$. This means that it is the graph 
whose vertices are the equivalence classes for $\sim_{\square}$ of vertices of $\mathcal{U}_G[a]$,
and there is an edge between two classes if there is an edge between two elements of these classes. We also denote by $\pi_a$ the projection from $\mathcal{U}_G[a]$ to  $\quadcover{G}[a]$.
\end{notation}

Let us see that the graphs $\quadcover{G}[a]$, $a \in V_G$ are all isomorphic.

\begin{lemma}
For every $a, b$ neighbors in $G$ and $p,q \in V_{\mathcal{U}_G[b]}$ which differ by a square, we have that $\psi_{b \mapsto a}(p)$ and $\psi_{b \mapsto a}(q)$ also differ by a square.
\end{lemma}

\begin{proof}
There exist a square $s$ and some $k \le l(q)$ such that $q = \varphi(p \oplus_k s)$. Since removing spines in any order give the same result, we have:
\[\psi_{b \mapsto a}(q) = \varphi(aq) = \varphi(a\varphi(p \oplus_k s)) = \varphi(a (p\oplus_k s)) = \varphi (ap \oplus_{k+1} s).\]
We now distinguish three cases.
\begin{enumerate}
\item If $p_1 \neq a$, then $\psi_{b \mapsto a}(p) = ap$ so $\varphi (ap \oplus_{k+1} s) = \varphi(\psi_{b \mapsto a}(p) \oplus_{k+1} s)$.

\item If $p_1 = a$ and $k>0$, then $ap \oplus_{k+1} s = ap_0 (\psi_{b \mapsto a}(p) \oplus_{k-1} s)$ and, since spines can be removed in any order, 
$\varphi (ap \oplus_{k+1} s) = \varphi(\psi_{b \mapsto a}(p) \oplus_{k-1} s)$.

\item If $p_1 = a$ and $k=0$, then $\psi_{b \mapsto a}(q) = \varphi(a s_0 s_1 s_2 s_3 s_0\odot p)$. We again distinguish three cases:
\begin{enumerate}
\item If $s_3 = a$, then $\psi_{b \mapsto a}(q) = \varphi(a s_0 s_1 s_2 \varphi(ap)) = \varphi(\psi_{b \mapsto a}(p) \oplus_0 a s_0 s_1 s_2 s_3)$.
\item If $s_1 = a \neq s_3$, then $\varphi(a s_0 s_1 s_2 s_3 p) = a s_2 s_3 p = \varphi(\psi_{b\to a}(p) \oplus_0 a s_2 s_3 b a)$.
\item If $s_1 \neq a \neq s_3$, then $\varphi(a s_0 s_1 s_2 s_3 p) = a s_0 s_1 s_2 s_3 p$. In that case, $\varphi(\psi_{b \mapsto a}(q) \oplus_1 s^{-1}) = \varphi(ap) = \psi_{b \mapsto a}(p)$.

\end{enumerate}
\end{enumerate}
In all three cases, $\psi_{b \mapsto a}(p)$ and $\psi_{b \mapsto a}(q)$ differ by a square.
\end{proof}

\begin{corollary}\label{cor.morphism.quotient}
For every $a, b$ neighbors in $G$ and $p,q \in V_{\mathcal{U}_G[b]}$ such that $p \sim_{\square} q$, we have that $\psi_{b \mapsto a}(p) \sim_{\square} \psi_{b \mapsto a}(q)$.
\end{corollary}

A direct consequence is that for all $a,b$ neighbors in $G$, we can define a morphism $\overline{\psi}_{a \mapsto b} : \quadcover{G}[a] \rightarrow \quadcover{G}[b]$ by setting $\overline{\psi}_{a \mapsto b} (\pi_a(p)) = \pi_b\left(\psi_{a\mapsto b}(p)\right)$ for every $p \in V_{\mathcal{U}_G[a]}$. Furthermore: 

\begin{lemma}
The graphs $\quadcover{G}[a]$, $a \in V_G$, are all isomorphic. 
\end{lemma}

\begin{proof}
As a direct consequence of Lemma~\ref{lemma.morphism}, 
the morphism $\overline{\psi}_{a \mapsto b} \circ \overline{\psi}_{b \mapsto a}$ is the identity of $V_{\quadcover{G}[a]}$.
\end{proof}

\begin{definition}\label{definition.square.cover}
We call \textbf{square cover} of $G$ and denote by $\quadcover{G}$ the isomorphic class of the graphs $\quadcover{G}[a]$, for $a \in V_G$.
\end{definition}

In the same way as $\mathcal{U}_G$, we represent $\quadcover{G}$ as an unlabeled graph which admits labelings $\mathcal{U}_G^{\square}[a]$, $a \in V_G$. Some examples are given on Figure~\ref{figure.example.cover}. Notice that the example $\textbf{(b)}$ is isomorphic to its square cover. Proposition~\ref{proposition.equal.square} below states that this is true for any square-decomposable graph. On the other hand a graph which is not square-decomposable may have a finite or infinite square cover (examples $\textbf{(a)}$ and $\textbf{(c)}$). 

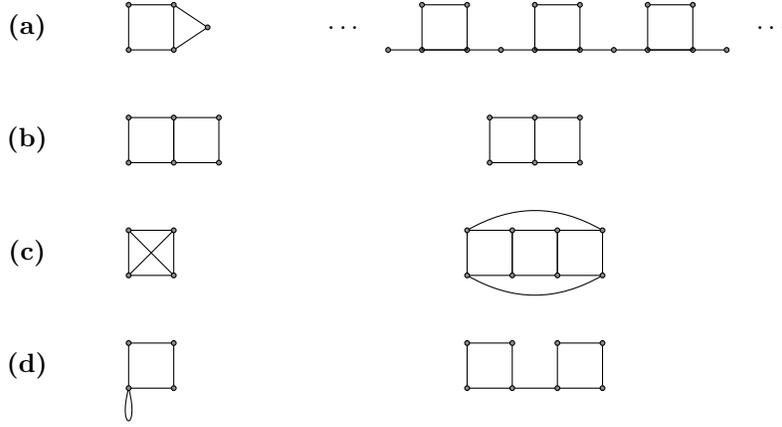
\begin{figure}[h!]
\begin{center}
\begin{tikzpicture}[scale=0.3]
\begin{scope}[yshift=1cm]
\node at (-4.5,1) {\textbf{(a)}};
\draw (0,0) rectangle (2,2);
\draw (2,2) -- (3.5,1) -- (2,0);
\draw[fill=gray!90] (0,0) circle (3pt);
\draw[fill=gray!90] (2,2) circle (3pt);
\draw[fill=gray!90] (0,2) circle (3pt);
\draw[fill=gray!90] (2,0) circle (3pt);
\draw[fill=gray!90] (3.5,1) circle (3pt);
\end{scope}

\begin{scope}[xshift=6cm,yshift=1cm]
\draw (5.5,0) -- (20.5,0);
\draw[fill=gray!90] (5.5,0) circle (3pt);
\draw[fill=gray!90] (7,0) circle (3pt);
\draw[fill=gray!90] (9,0) circle (3pt);
\draw[fill=gray!90] (10.5,0) circle (3pt);
\draw[fill=gray!90] (12,0) circle (3pt);
\draw[fill=gray!90] (14,0) circle (3pt);
\draw[fill=gray!90] (15.5,0) circle (3pt);
\draw[fill=gray!90] (17,0) circle (3pt);
\draw[fill=gray!90] (19,0) circle (3pt);
\draw[fill=gray!90] (20.5,0) circle (3pt);

\draw (7,0) rectangle (9,2);
\draw[fill=gray!90] (7,2) circle (3pt);
\draw[fill=gray!90] (9,2) circle (3pt);

\draw (12,0) rectangle (14,2);
\draw[fill=gray!90] (12,2) circle (3pt);
\draw[fill=gray!90] (14,2) circle (3pt);

\draw (17,0) rectangle (19,2);
\draw[fill=gray!90] (19,2) circle (3pt);
\draw[fill=gray!90] (17,2) circle (3pt);

\node at (22.5,1) {$\hdots$};
\node at (3.5,1) {$\hdots$};
\end{scope}

\begin{scope}[yshift=-4cm]
\node at (-4.5,1) {\textbf{(b)}};
\draw (0,0) rectangle (2,2);
\draw (2,2) rectangle (4,0);
\draw[fill=gray!90] (0,0) circle (3pt);
\draw[fill=gray!90] (2,2) circle (3pt);
\draw[fill=gray!90] (0,2) circle (3pt);
\draw[fill=gray!90] (2,0) circle (3pt);
\draw[fill=gray!90] (4,0) circle (3pt);
\draw[fill=gray!90] (4,2) circle (3pt);

\begin{scope}[xshift=16cm]
\draw (0,0) rectangle (2,2);
\draw (2,2) rectangle (4,0);
\draw[fill=gray!90] (0,0) circle (3pt);
\draw[fill=gray!90] (2,2) circle (3pt);
\draw[fill=gray!90] (0,2) circle (3pt);
\draw[fill=gray!90] (2,0) circle (3pt);
\draw[fill=gray!90] (4,0) circle (3pt);
\draw[fill=gray!90] (4,2) circle (3pt);
\end{scope}
\end{scope}

\begin{scope}[yshift=-9cm]
\node at (-4.5,1) {\textbf{(c)}};
\draw (0,0) rectangle (2,2);
\draw (0,0) -- (2,2) (0,2) -- (2,0);
\foreach \x/\y in {0/0, 0/2, 2/0, 2/2}
{
  \draw[fill=gray!90] (\x,\y) circle (3pt);
}

\begin{scope}[xshift=15cm]
\draw (0,0) rectangle (2,2);
\draw (2,0) rectangle (4,2);
\draw (4,0) rectangle (6,2);
\draw (0,2) edge[bend left] (6,2);
\draw (0,0) edge[bend right] (6,0);
\foreach \x/\y in {0/0, 0/2, 2/0, 2/2, 4/0, 4/2, 6/0, 6/2}
{
  \draw[fill=gray!90] (\x,\y) circle (3pt);
}
\end{scope}
\end{scope}

\begin{scope}[yshift=-14cm]
\node at (-4.5,1) {\textbf{(d)}};
\draw (0,0) rectangle (2,2);
\draw (0,0) edge[loop below] node{} (0,0);
\foreach \x/\y in {0/0, 0/2, 2/0, 2/2}
{
  \draw[fill=gray!90] (\x,\y) circle (3pt);
}

\begin{scope}[xshift=15cm]
\draw (0,0) rectangle (2,2);
\draw (4,0) rectangle (6,2);
\draw (2,0) -- (4,0);
\foreach \x/\y in {0/0, 0/2, 2/0, 2/2, 4/0, 4/2, 6/0, 6/2}
{
  \draw[fill=gray!90] (\x,\y) circle (3pt);
}
\end{scope}
\end{scope}

\end{tikzpicture}
\end{center}
\caption{Some finite graphs (on the left) and their square cover (on the right).\label{figure.example.cover}}
\end{figure}


\begin{notation}
 Since two walks which differ by a square have the same endpoint, we can define  $\eta :  \bigcup_a \quadcover{G}[a] \to G$ by setting $\eta(\pi_a(p)) := p_{l(p)}$ for any non backtracking walk $p$ which begins with $a$. We use the same notation for walks: for all $a$ and all walk $q$ on $\quadcover{G}[a]$, we set $\eta(q) := \eta(q_0) \ldots \eta(q_{l(q)})$.
 \end{notation}

We now arrive at the goal we set in Section~\ref{section.universal}: lifting configurations of $X_G$ to $X_{\quadcover{G}}$; such lifting allows us to relate block-gluing properties of $X_G$ and $X_{\quadcover{G}}$. 

\begin{proposition}\label{prop.square.to.graph}
For every $x \in X_G$ and $\textbf{i}_0 \in \mathbb{Z}^2$, there exists a configuration $z$ of $X_{\quadcover{G}[x_{\boldsymbol{i}_0}]}$ such that $z_{\boldsymbol{i}_0} = \pi_{ x_{\boldsymbol{i}_0}}(x_{\boldsymbol{i}_0})$ and for all $\boldsymbol{i} \in \mathbb{Z}^2$, $\eta(z_{\boldsymbol{i}})=x_{\boldsymbol{i}}$.
\end{proposition}

\begin{proof}
For all $\textbf{i} \in \mathbb{Z}^2$, take any walk $\textbf{i}_0 \dots \textbf{i}_m$ in $\mathbb{Z}^2$ such that $\textbf{i}_m = \textbf{i}$ , and define $z_{\boldsymbol{i}} = \pi_{ x_{\boldsymbol{i}_0}}(x_{\textbf{i}_0}  \dots  x_{\textbf{i}_m})\in\quadcover{G}[ x_{\boldsymbol{i}_0}]$. Taking another walk in $\mathbb Z^2$ would yield the same equivalence class $z_{\boldsymbol{i}}$, since the corresponding walks in $G$ are equivalent for $\sim_\square$. Therefore $z$ is well-defined and satisfies the requirements.
\end{proof}

\begin{proposition}\label{proposition.square.to.gamma}
Let $G$ be a graph such that $\quadcover{G}$ is finite. Then $\gamma_{G} = O(\gamma_{\quadcover{G}})$.
\end{proposition}

Note that, since $\quadcover{G}$ is always bipartite, $G$ and $\quadcover{G}$ do not necessarily have the same phase (for block-gluing). Nevertheless, Proposition~\ref{prop:block-gluing} still applies.

\begin{proof}
By Proposition~\ref{prop:block-gluing} and Lemma~\ref{lemma.even}, it is sufficient to prove that 
\[\textrm{diam}(\Delta_G^{2n}) = O\left( \textrm{diam}(\Delta_{\quadcover{G}}^{2n})\right).\]

For all $n$, set $\Lambda_G^{n} = \{s^{n} : s\text{ is a spine on }G\}$. It is straightforward that $\textrm{diam}(\Lambda_G^{n}) \le \textrm{diam}(G)$. It is thus sufficient to prove that
\[\max_{p \in \Delta_G^{2n}} \min_{q \in \Lambda_G^{n}} d_G(p,q) \leq \textrm{diam}(\Delta_{\quadcover{G}}^{2n}),\]
where the left-hand term is related to the Hausdorff distance between $\Delta_G^{2n}$ and $\Lambda_G^{n}$.

We denote $(\pi_{p_0}\circ \varphi)^{\ast}$ the function that, to a walk $p$ on $G$, associates the sequence $(\pi_{p_0}\circ \varphi(p_0 \ldots p_i))_{i \le l(p)}$.

Let $p$ be a walk of length $2n$ on $G$ and $q$ an element of $\Lambda_G^{n}$ which begins on $p_0$. There is a walk $(p^{(i)})_{0\le i\le k}$ of length $k$ in $\Delta_{\quadcover{G} [p_0]}^{2n}$ between $(\pi_{p_0}\circ \varphi)^{\ast} (p)$ and $(\pi_{p_0}\circ \varphi )^{\ast}(q)$ for some $k\le \textrm{diam}(\Delta_{\quadcover{G} [p_0]}^{2n})$. Then the walk $(\eta(p^{(i)}))_{0\le i\le k}$ is a walk in $\Delta_G^{2n}$ from $\eta((\pi_{p_0}\circ \varphi)^{\ast}(p)) = p$ to $\eta((\pi_{p_0}\circ \varphi)^{\ast}(q)
) = q$. 
From this we deduce that $d_G(p, q) \leq \textrm{diam}(\Delta_{\quadcover{G}}^{2n})$. Since $p$ can be chosen arbitrarily and $q$ chosen according to $p$, this implies that $\max_{p \in \Delta_G^{2n}} \min_{q \in \Lambda_G^{n}} d_G(p,q) \leq \textrm{diam}(\Delta_{\quadcover{G}}^{2n})$, and the proposition follows.
\end{proof}

\begin{lemma}\label{lemma:equivalence:walks}
Let us assume that $G$ is square-decomposable. Then 
for all $p,q$ non-backtracking walks on $G$ such that $p_0=q_0$ and $p_{l(p)} = q_{l(q)}$, $p \sim_{\square} q$.
\end{lemma}

\begin{proof}
Let us prove this by induction on the area of the cycle $p \odot q^{-1}$. When this area is equal to $0$, $p$ is equal to $q$ and they are therefore equivalent for $\sim_{\square}$. Let us assume that the statement is proved whenever $p \odot q^{-1}$ has area $\le m$, and fix $p,q$ for which this cycle has area $m+1$. Consider some square $s$ and integer $k$ such that $\varphi((p \odot q^{-1}) \oplus_k s)$ has area $m$. Let us denote by $p'$ and $q'$ the following paths:
\begin{enumerate}
    \item if $k\le l(p)$, $p' = p\oplus_k s$ and $q'=q$;
    \item if $k> l(p)$, $p'=p$ and $q' = q \oplus_{l(q)-(k-l(p))} s^{-1}$.
\end{enumerate}
By induction, $p' \sim_{\square} q'$. Therefore, by definition of $p',q'$, $p \sim q$.
\end{proof}

\begin{proposition}\label{proposition.equal.square}
If $G$ is square-decomposable, $G$ and $\quadcover{G}$ are isomorphic.
\end{proposition}

\begin{proof}
It is sufficient to see that in this case $G$ is 
isomorphic to $\quadcover{G}[a_0]$ for $a_0$ some vertex  of $G$. For every $a$ in $G$, choose some non-backtracking simple walk $p_a$ from $a_0$ to $a$. The elements of $\pi_{a_0}(p_a)$ are walks which begin at $a_0$ and end at $a$. As a consequence, all the classes $\pi_{a_0}(p_a)$, $a \in V_G$ are different. Furthermore, because $G$ is square-decomposable, Lemma~\ref{lemma:equivalence:walks} implies that 
every equivalence class is equal to some $\pi_{a_0}(p_a)$. 

Furthermore for all $a,b$ there is an edge between $a$ and $b$ in $G$ if and only if there is an edge between 
$\pi_{a_0}(p_a)$ and $\pi_{a_0}(p_b)$ in $\quadcover{G}$, which yields the statement.
\end{proof}

\begin{proposition}\label{proposition.square.cover.decomposable}
The square cover of $G$ is square-decomposable.
\end{proposition}

\begin{proof}
We will use the following: if $c$ is a cycle on $\quadcover{G}$, then $\eta(c)$ is a cycle, and $c$ has a backtrack iff $\eta(c)$ has a backtrack as well.


Take any vertex $v$ of $\quadcover{G}$. By repeated application of morphisms $\overline{\psi}$, we can choose a labeling of $\quadcover{G}$ so that $v$ is the class for $\sim_\square$ of the empty walk.
Let $w$ be a walk on $G$ such that $w_0 = \eta(v)$. We denote by $\eta^{-1}_v(w)$ the walk such that for all $i \le l(w)$:
\[\eta^{-1}_v(w)_i\text{ is the class for }\sim_\square\text{ of the walk } \varphi(w_0\dots w_i).\]
It is straightforward that for any walk $w'$ on $\quadcover{G}$, $\eta^{-1}_{w'_0}\circ\eta(w') = w'$.



Let us turn to the proof. Consider a cycle $c$ on $\quadcover{G}$. We will prove that it is square-decomposable.
\begin{enumerate}
\item \textbf{The cycle $\eta(c)$ is square-decomposable:}

Since $c$ is a non-backtracking cycle, $\eta(c)$ is also a non-backtracking cycle. 
Furthermore, since $c_0 = c_{l(c)}$, there exists $p$ in $c_0$ such that $\varphi(p \eta(c)) \in c_0$, therefore 
$\varphi(p \eta(c)) \sim_{\square} p$.
Then we can use Corollary~\ref{cor.morphism.quotient} repeatedly to have 
\[\psi_{p_{l(p)-1} \mapsto p_{l(p)}} \circ \cdots \circ \psi_{p_0 \mapsto p_1} (\varphi(p \eta(c))) \sim_{\square} \psi_{p_{l(p)-1} \mapsto p_{l(p)}} \circ \cdots \circ \psi_{p_0 \mapsto p_1} (p).\]

This is rewritten into $\eta(c) \sim_{\square} p_{l(p)}$ (empty cycle), which means that $\eta(c)$ is square-decomposable. Let us consider $(q^{(i)})_{0\le i\le l}$ a square decomposition of $\eta(c)$.

\item \textbf{The sequence $(\eta^{-1}_{c_0}(q^{(i)}))_{0\le i\le l}$ is a square decomposition of $c$:} 

It is clear that $\eta^{-1}_{c_0}(\eta(c)) = c$ and $\eta^{-1}_{c_0}(\eta(c_0)) = c_0$ (empty cycle). Therefore, it is sufficient to prove that for all $i$ such that $0\leq i < l$, $\eta^{-1}_{c_0}(q^{(i)})$ and $\eta^{-1}_{c_0}(q^{(i+1)})$ differ by a square.

Considering such an integer $i$, we can assume that $q^{(i+1)} = q^{(i)} \oplus_k s$ for some square $s$ and index $k$ (the other possible case is processed similarly). For any vertex $v$ of $\quadcover{G}$ such that $\eta(v) = s_0$, we have that $\eta^{-1}_{v}(s)$ is a square in $\quadcover{G}$. Indeed, it is non-backtracking, and it is a cycle since $v$ and $v\odot s$ differ by a square.

It is straightforward to check that $\eta^{-1}_{c_0}(q^{(i+1)}) = \eta^{-1}_{c_0}(q^{(i)}) \oplus_k \eta^{-1}_{v}(s)$, where $v =  \eta^{-1}_{c_0}(q^{(i)})_k$. Since $\eta^{-1}_{v}(s)$ is a square, we have that $\eta^{-1}_{c_0}(q^{(i)})$ and $\eta^{-1}_{c_0}(q^{(i+1)})$ differ by a square. This concludes the proof.\qedhere
\end{enumerate}
\end{proof}

\subsection{ When $|\quadcover{G}| = +\infty$, the gap function is linear\label{section.square.cover.infinite}}

\cite{CM18} works in the context of four-cycle-hom-free graphs. A graph is four-cycle hom-free if and only if it has no non-backtracking cycle of length four, which means that the universal cover and the square cover are equal. Indeed, in this case the square cover is the quotient of the universal cover by an empty set of relations.

In this context, Corollary 5.6 in \emph{op.cit.} restated in our notation is:

\begin{theorem}[\cite{CM18}]\label{thm-chandgotia}
Let $G$ be a four-cycle-hom-free graph. $\gamma_G = \Theta(1)$ if and only if $\mathcal{U}_{G}$ is finite.
\end{theorem} 

Starting from this section, we extend this work by characterising the equivalence classes of $\gamma_G$ for general graphs by using the square cover. While Chandgotia and Marcus' result holds in any dimension, we restrict our attention to dimension two for the reasons mentioned in the introduction.  However, the proof of the next theorem specifically can be easily extended in any dimension.

We generalise the ``infinite'' case of Theorem~\ref{thm-chandgotia} as follows:

\begin{theorem}\label{thm:main1}
If the square cover $\quadcover{G}$ of $G$ is infinite, then $\gamma_G(n) = \Theta(n)$. 
\end{theorem}

\begin{proof}
Let us fix some vertex $a$ of $G$.
Since $\quadcover{G}$ is infinite and the degree of every vertex is less than $|V_G|$, for all $n$ there exists a path $p_n$ on $G$ such that $\pi_a(p_n)$ is at distance $2n$ from $\pi_a(a)$ in $\quadcover{G}[a]$. Let $a = a_0, \ldots , a_{2n}$ be walks on $G$ such that $\pi_a (a_0) \cdots \pi_a (a_{2n})$ is a walk (of length $2n$) from $\pi_a(a)$ to $\pi_a(p_n)$ in $\quadcover{G}[a]$. We denote $u := a w_1 \cdots w_{2n}$ its image for $\eta$, and $v:= (a w_1)^{n} a$. Let us prove that $d_{G} (u,v)\geq n$. 

Take some $x \in X_G$ such that $x|_{\llbracket 0, 2n \rrbracket \times \{0\}} = u$ and $x|_{\llbracket 0, 2n \rrbracket \times \{k\}} = v$ for some $k<n$. We know by Proposition~\ref{prop.square.to.graph} that there exists a unique configuration $z$ of $X_{\quadcover{G}}$ such that $z_{0,0} 
= \pi_a (a)$ and $\eta(z_{\textbf{i}}) = x_{\textbf{i}}$ for all $\textbf{i} \in \mathbb{Z}^2$. As a consequence, $z_{2n,0}$ is the class $\pi_a (p_n)$, and since $\pi_a (v)=\pi_a (a)$ (empty cycle), $z_{2n,k} = z_{0,k}$. By triangular inequality, 
\[d_{\quadcover{G}} (z_{0,0}, z_{2n,0}) \leq d_{\quadcover{G}} (z_{0,0}, z_{0,k}) + d_{\quadcover{G}} (z_{0,k}, z_{2n,k}) + d_{\quadcover{G}} (z_{2n,k}, z_{2n,0}) \leq 2k.\] 

It follows that $k \geq n$. We conclude that $\gamma_G(2n) \ge n$ for all $n \ge 1$, which implies that 
$\gamma_G(n) = \Theta(n)$.
\end{proof}

\begin{figure}
\begin{center}
$G: $
\begin{tikzpicture}[every node/.style={draw,circle, minimum size=.6cm, inner sep=0pt}, baseline={([yshift=-.5ex]current bounding box.center)}]
\node (a) at (0,0) {a};
\node (b) at (1.6,0) {b};
\node (c) at (.8,.8) {c};
\draw (a) -- (b) -- (c) -- (a);
\end{tikzpicture}

\begin{tikzpicture}[every node/.style={}, baseline={([yshift=-.5ex]current bounding box.center)}, scale=0.9]
\draw[step = .8, gray, thin] (0,0) grid (4.8,4);
\node[fill=blue!20!white] at (.4, 3.6) {a};
\node[fill=blue!20!white] at (1.2, 3.6) {b};
\node[fill=blue!20!white] at (2, 3.6) {c};
\node[fill=blue!20!white] at (2.8, 3.6) {a};
\node[fill=blue!20!white] at (3.6, 3.6) {b};
\node[fill=blue!20!white] at (4.4, 3.6) {c};
\node[fill=red!20!white] at (.4, 0.4) {a};
\node[fill=red!20!white] at (1.2, 0.4) {b};
\node[fill=red!20!white] at (2, 0.4) {a};
\node[fill=red!20!white] at (2.8, 0.4) {b};
\node[fill=red!20!white] at (3.6, 0.4) {a};
\node[fill=red!20!white] at (4.4, 0.4) {b};
\end{tikzpicture}
$\xrightarrow{\text{lift to $X_{\quadcover{G}}$}}$
\begin{tikzpicture}[every node/.style={},baseline={([yshift=-.5ex]current bounding box.center)}, scale=0.9]

\draw[step = .8, gray, thin] (0,0) grid (4.8,4);
\node at (.4, 3.6) {0};
\node at (1.2, 3.6) {1};
\node at (2, 3.6) {2};
\node at (2.8, 3.6) {3};
\node at (3.6, 3.6) {4};
\node at (4.4, 3.6) {5};

\node at (.4, 0.4) {$k$};
\node[scale = 0.8] at (1.2, 0.4) {$k\pm 1$};
\node at (2, 0.4) {$k$};
\node[scale = 0.8] at (2.8, 0.4) {$k\pm 1$};
\node at (3.6, 0.4) {$k$};
\node[scale = 0.8] at (4.4, 0.4) {$k\pm 1$};
\path[->, thick] (0,4.2) edge node[midway, above] {$2n$} (4.8, 4.2);
\path[<-, thick] (-.2, -.2) edge (-.2, 4.2);
\path[->, thick] (5, -.2) edge node[midway, right] {$\geq n$}(5, 4.2);
\path[->, thick] (-.2,-.2) edge node[midway, below] {$\approx 0$} (5, -.2);
\end{tikzpicture}
\end{center}
\caption{Illustration of the proof of Theorem~\ref{thm:main1} on an example.}
\end{figure}
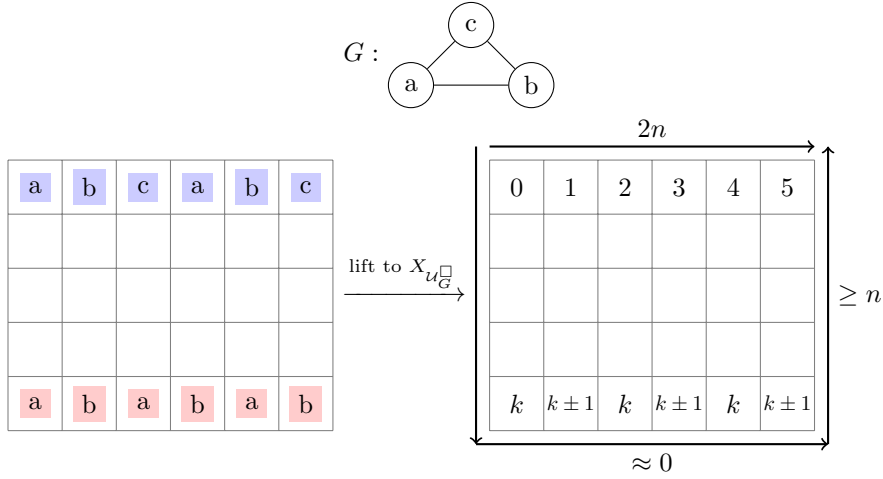

\section{Writing cycles as trees of simple cycles\label{section.cactus}}
We use a representation of cycles on $G$ as rooted finite ordered trees whose vertices are labeled with simple cycles of $G$ satisfying some conditions that allow to 'glue' them together, that we call \emph{cacti}. This construction is used when proving that some Hom shifts have logarithmic gluing gap (Theorem~\ref{theorem.square.dismantable.to.log}) by applying some transformations "in parallel" to different parts of a large cycle.

\subsection{Definition of a cactus\label{section.cactus.rep}}

Here is an inductive definition of a cactus. Let us recall (Notation~\ref{notation.simple.cycle}) that $C_G^0$ denotes the set of simple cycles on $G$.

\begin{definition}
For $n \geq 1$, a \textbf{cactus of depth} $\boldsymbol{n}$ is a triple $C=(\xi,s,\chi)$ such that $\xi \in \mathcal{C}^0_G$, $s=(s_1, \dots , s_d)$ is a sequence of cacti whose depth is strictly less than $n$, and $\chi = (\chi_1, \dots , \chi_d)$ is a non-decreasing sequence of non-negative integers such that: 

\begin{itemize}
\item For all $i$, $\chi_i \le l(\xi)$ and $\xi(s_i)_0 = \xi_{\chi_i}$.
\item At least one element of $s$ has depth exactly $n-1$. 
\end{itemize}

A \textbf{cactus of depth} $\boldsymbol{1}$, also called a \textbf{leaf}, must have empty sequences for $s$ and $\chi$. In this case we identify the leaf $C$ and $\xi \in \mathcal{C}_G^0$.

We use the notation $(\xi(C),s(C),\chi(C)):=C$, and denote by $d(C)$ the common length of $s(C)$ and $\chi(C)$. Furthermore, we denote the depth of $C$ by $n(C)$.
\end{definition}

\begin{definition}
A \textbf{cactus forest} is a sequence of cacti $(C_1, \ldots C_k)$ such that $\xi(C_j)_0$ does not depend on $j$. Its depth, denoted by $n(C_1,\ldots,C_k)$, is equal to $\max_i n(C_i)$.
\end{definition}

\begin{notation}
For every cactus $C$ and every $k \le n(C)$, we call \textbf{$k$-th level} of $C$ and denote $\ell_k(C)$ the set of cacti defined inductively by:
\[\begin{array}{ll}
\ell_1(C)=\{C\}\\
\ell_k(C) = \bigcup_{1\le i \le d(C)} \ell_{k-1}(s(C)_i)\quad \text{when }k \ge2. 
\end{array}\]

Furthermore, for all $k > n(C)$, we set $\ell_k(C) = \emptyset$.
For a cactus forest $(C_1,\ldots,C_l)$, and $k \le n(C_1,\ldots,C_l)$, its $k$-th level $\ell_k(C_1,\ldots,C_l)$ is defined as 
\[\ell_k(C_1,\ldots,C_l) = \bigcup_{j} \ell_k(C_j).\]

\end{notation}

\begin{example} \label{example.cactus}A cactus can be represented as labeled rooted tree, although we will not use such a representation formally:

\begin{figure}[h!]
\begin{center}
\begin{tikzpicture}[every node/.style={draw,circle, minimum size=.6cm, inner sep=0pt}, baseline={([yshift=-.5ex]current bounding box.center)}]
\node (a) at (0,0) {a};
\node (b) at (1.6,0) {b};
\node (c) at (.8,.8) {c};
\draw (a) -- (b) -- (c) -- (a);
\path (a) edge[loop left, looseness=5, in=150, out=210] (a);
\end{tikzpicture}
\hspace{2cm}
\begin{tikzpicture}[scale=0.7, baseline={([yshift=-.5ex]current bounding box.center)}]
\draw (0,0) edge node[above]{1} (2,1); 
\draw (0,0) edge node[below]{0} (2,-1);
\draw (4,0) edge node[below]{1} (2,1);
\draw (2,1) edge node[below]{2} (4,2);
\draw[fill=gray!80] (0,0) circle (3pt);
\draw[fill=gray!80] (2,1) circle (3pt);
\draw[fill=gray!80] (2,-1) circle (3pt);
\draw[fill=gray!80] (4,0) circle (3pt);
\draw[fill=gray!80] (4,2) circle (3pt);
\node at (0,-0.5) {$abca$};
\node at (2,-1.5) {$aa$};
\node at (2,1.5) {$bcab$};
\node at (4,-0.5) {$cac$};
\node at (4,2.5) {$abca$};
\end{tikzpicture}.
\end{center}
\caption{A graph $G$ and a cactus of depth $3$ on $G$. Every subtree is a cactus $C$ whose root is labeled by $\xi(C)$, whose children are $s(C)$ (ordered from bottom to top) and where the edge to each child is labeled by the corresponding $\chi_i(C)$.}
\end{figure}
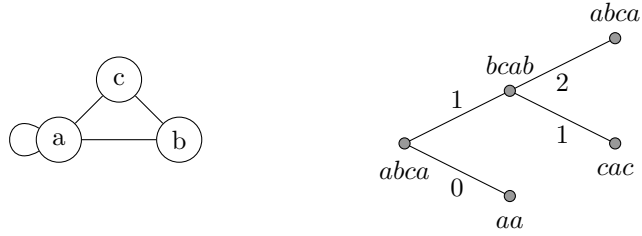
\end{example}

\subsection{Cycle $\pi(C)$ \textit{encoded} by a cactus forest $C$}

The purpose of cacti is to encode cycles on $G$. It is clear how a cactus $C$ of depth $1$ encodes 
$\xi(C)$. The cycle encoded by a cactus $C$ of depth $n>1$ is obtained by `plugging' in $\xi(C)$ the cycles 
encoded by the cacti in the sequence $s(C)$ in order, on positions of $\xi(C)$ determined by the sequence $\chi(C)$. We provide a formal definition below.

\begin{notation}
Let us recall that for two cycles $c$ and $c'$ and $k \le l(c)$ such that $c_k = c'_0$, the cycle $c {\oplus}_{k} c'$ is defined by:
\[c {\oplus}_{k} c' = c_0 \hdots c_{k-1} c' c_{k+1} \hdots c_{l(c)}.\]
More generally, consider a sequence of cycles $c^{(0)}, c^{(1)}, \dots ,c^{(m)}$ with $m \ge 2$ and a non-decreasing sequence of integers $k_1, \dots , k_m \le l(c)$ such that for all $0 < j \le m$, 
$c^{(0)}_{k_j} = c^{(j)}_{0}$. We define inductively a cycle 
$c^{(0)} \oplus_{k_1} c^{(1)} \ldots \oplus_{k_m} c^{(m)}$ by:
\[c^{(0)} \oplus_{k_1} c^{(1)} \ldots \oplus_{k_m} c^{(m)} = (c^{(0)} \oplus_{k_1} c^{(1)} \ldots \oplus_{k_{m-1}} c^{(m-1)}) \oplus_{k_m+\sum_{j=1}^{m-1} l(c^{(j)})} c^{(m)}\]
\end{notation}

\begin{remark}
This reflects what we mean by `plugging' successively the cycles $c^{(j)}$ in $c$ at positions $k_j$. 
\end{remark}

\begin{notation}Every cactus $C$ is said to \emph{encode} a cycle $\pi(C)$, which we define inductively.
When $C$ is a leaf, $\pi(C) := C$. For any depth $n \ge2$, and $C$ a cactus of depth $n$, we set: 
\[\pi(C) = \xi(C)\oplus_{\chi(C)_0}\pi(s(C)_0) \oplus_{\chi(C)_1}\pi(s(C)_1)\oplus\cdots \oplus_{\chi(C)_{d(C)}}\pi(s(C)_{d(C)}).
\]
As well, a cactus forest $(C_1,\ldots,C_k)$ encodes the 
cycle \[\pi(C_1) \odot \pi(C_2) \ldots \odot \pi(C_k),\]
which we denote by $\pi(C_1,\ldots,C_k)$.
\end{notation}

\begin{example}
For the cactus $C$ defined in Example~\ref{example.cactus}, the cycle $\pi(C)$ is: 
\[\pi(C) = aabcacabcabca.\]
\end{example}

\subsection{Encoding cycles by cactus forests of bounded depth\label{section.bounded.depth}}

A cycle may be encoded by more than one cacti:

\begin{example}
Let $C$ be the cactus defined in Example~\ref{example.cactus}. The cycle $\pi(C)$ is encoded by another cactus $C' \neq C$ (which means that $\pi(C') = \pi(C)$):
\begin{center}
\begin{tikzpicture}[every node/.style={draw,circle, minimum size=.6cm, inner sep=0pt}, baseline={([yshift=-.5ex]current bounding box.center)}]
\node (a) at (0,0) {a};
\node (b) at (1.6,0) {b};
\node (c) at (.8,.8) {c};
\draw (a) -- (b) -- (c) -- (a);
\path (a) edge[loop left, looseness=5, in=150, out=210] (a);
\end{tikzpicture}
\hspace{2cm}
\begin{tikzpicture}[scale=0.7, baseline={([yshift=-.5ex]current bounding box.center)}]
\draw (0,0) edge node[above]{2} (2.5,1.2); 
\draw (0,0) edge node[above, near end]{2} (2.5,.35);
\draw (0,0) edge node[below, near end]{2} (2.5,-.35);
\draw (0,0) edge node[below]{0} (2.5,-1.2);
\draw[fill=gray!80] (0,0) circle (3pt);
\draw[fill=gray!80] (2.5,1.2) circle (3pt);
\draw[fill=gray!80] (2.5,-1.2) circle (3pt);
\draw[fill=gray!80] (2.5,.35) circle (3pt);
\draw[fill=gray!80] (2.5,-.35) circle (3pt);
\node at (0,-0.5) {$abca$};
\node at (2.5,-1.7) {$aa$};
\node at (3.25,-.5) {$cac$};
\node at (3.25,.5) {$cabc$};
\node at (2.5,1.7) {$cabc$};
\end{tikzpicture}
\end{center}
\end{example}

However, we can encode any cycle with a cactus forest which has the nice property of having bounded depth: 

\begin{lemma}\label{lemma.cactus}
For any cycle $c$ on a graph $H$, there exists a  cactus forest $(C_1,\ldots,C_k)$ such that $\pi(C_1,\ldots,C_k) =c$  with $n(C_1,\ldots,C_k) \le |V_H|$.
\end{lemma}

\begin{proof}
Let us consider a graph $H$ and let us prove the statement for this graph, by induction on $|V_H|$. 
When $|V_H|=1$, since we assumed that all graph considered are connected, $H$ consists in a unique vertex $a$ with a self loop. Therefore the statement is straightforward, since all the cycles on $H$ are of the form $a \ldots a$. 
Let us assume that we have proved the lemma whenever $|V_H| \le n$ for some integer $n \geq 1$ and assume that $|V_H|=n+1$. Let us consider a cycle $c$ on $H$. 

Without loss of generality we can assume that $c_j = c_0$ implies that $j=0$ or $j=l(c)$. Indeed we can write $c$ as a product for $\odot$ of cycles which satisfy this property. It is sufficient then to prove that each of these cycles is encoded by a cactus forest (in practice we will prove that it is encoded by a cactus) in order to prove that $c$ is encoded by a cactus forest.

It is then straightforward that there exist a simple cycle $d$ 
which begins and ends at $c_0$ and a sequence of cycles 
$d^{(1)}, \ldots , d^{(k)}$ in which $c_0$ does not appear, and integers $l_1, \ldots , l_k$ such that 
\[c = d \oplus_{l_1} d^{(1)} \ldots \oplus_{l_k} d^{(k)}.\]
The cycles $d^{(j)}$ can be seen as the maximal cycles 
which appear in $c$ in which $c_0$ does not appear.
Since the cycles $d^{(1)}, \ldots , d^{(k)}$
are on the subgraph $H'$ of $H$ on vertices $V_H \backslash \{c_0\}$, by induction each of these cycles is encoded by a cactus forest of depth no larger than $|V_H|-1$. This implies that 
$c$ is encoded by a cactus forest of depth no larger than $|V_H|$.\qedhere
\end{proof}

\begin{remark}
Notice that in Lemma~\ref{lemma.cactus}, the bound $|V_H|$ is tight only if there is a vertex of $H$ with a self-loop.
\end{remark}

\section{ When $|\quadcover{G}| < +\infty$, $\gamma_G(n) = O(\log(n))$\label{section.finite.square.cover}}

In this section, we prove that when the square-cover of $G$ is finite, $\gamma_G (n) = O(\log(n))$. As a consequence of Proposition~\ref{proposition.square.to.gamma} and Proposition~\ref{proposition.square.cover.decomposable} it is sufficient to prove that if $G$ is square-decomposable, then $\gamma_G(n) = O(\log(n))$. This is stated as Theorem~\ref{theorem.square.dismantable.to.log} in Section~\ref{section.main.finite.cover}.

To keep the exposition as clear as possible, we call `transformation' of a cycle $c$ into another cycle $c'$ of the same length $n$ a walk from $c$ to $c'$ in $\Delta_G^n$; the time taken by the transformation is the length of the walk. We will also say that, given a concatenation of cycles $c = c_1\odot\dots\odot c_n$ and transformations $\varphi_i : c_i \to c_i'$, these transformations can be applied `in parallel' if we can transform $c$ into $c' = c_1'\odot\dots\odot c_n'$ without taking more time than the longest transformation $\varphi_i$.

This part is structured as follows. In Section~\ref{section.additional.spines}, we provide a way to transform a cycle $c$ concatenated with additional spines into a power of a spine. Then in Section~\ref{section.spine.parallelisation} we show that this type of transformations can be applied in parallel to a certain extent. This `parallelization' will be a central tool in the proof of Theorem~\ref{theorem.square.dismantable.to.log}.

\subsection{Transforming cycles with additional spines\label{section.additional.spines}}

In this section, we show how to transform a cycle of the form 
$t^k \odot u$, where $u$ is a cycle and $t$ is a spine on $u_0$, into $t^{k + l(u)/2}$. In practice, we only need to consider the case $u=2\lambda_G$ $k=2\lambda_G$ and $u=c^{\lambda_G}$, where $\lambda_G$ is a characteristic of the graph $G$ defined below. Furthermore, we find a bound on the distance $d_G$ between the cycles $t^{2\lambda_G} \odot c^{\lambda_G}$ and $t^{2\lambda_G + \lambda_G l(c)/2}$. 

Although the distance $d_G$ is also a distance between cycles, we need, in order to parallelize these 
transformations, to use another distance that we denote by $d_G^{\mathcal{R}}$ (see Notation~\ref{notation.distance.R}). From the bounds obtained here on $d_G^{\mathcal{R}}$ for $t^k \odot u$, we are able to obtain bounds on $d_G^{\mathcal{R}}$ for larger words in which 
the transformations are executed in parallel (in  Section~\ref{section.spine.parallelisation}). 

Lemma~\ref{lem:comp-distances} provides a relation between $d_G$ and $d_G^{\mathcal{R}}$, which enables us to recover 
bounds on distances for $d_G$ from bounds on distances for $d_G^{\mathcal{R}}$.

The main result of this section is Corollary~\ref{cor:reduction2m}, which derives from the intermediate results in Lemma~\ref{lemma.small.steps} and Lemma~\ref{lemma.cycle.reduction}.

\begin{notation}\label{notation.distance.R}
For all cycles $c,c'$ on $G$ such that $l(c) = l(c')$ and $c_0 = c'_0$, denote $c \mathcal{R}_0 c'$ (resp. $c \mathcal{R}_1 c'$) when there exists a left (resp. right) shift of $c'$ which is neighbor of $c$ in the graph $\Delta_G^{l(c)}$. 
Also denote by $d_G^{\mathcal{R}}(c,c')$ the minimal $m$ such that there is a sequence $(c^{(k)})_{0 \le k \le m}$ of cycles such that $c^{(0)} = c$, $c^{(m)}=c'$ and for all $k < m$, $c^{(k)} \mathcal{R}_0 c^{(k+1)}$ or $c^{(k)} \mathcal{R}_1 c^{(k+1)}$.  
\end{notation}

\begin{lemma}\label{lem:comp-distances}
For all cycles $c,c'$ such that $l(c)=l(c')$, we have the inequality 
\[d_G(c,c') \le 2 d_G^{\mathcal{R}}(c,c').\]
\end{lemma}

\begin{proof}
First consider the case $d_G^{\mathcal{R}}(c,c') = 1$. Let us assume that $c \mathcal{R}_0 c'$ (the other case is symmetric). This means that there is a left shift of $c$ which is a neighbor of both $c'$ and $c$ in $\Delta_G^{l(c)}$. As a consequence $d_G(c,c') \le 2 = 2 d_G^{\mathcal{R}}(c,c')$.

For the general case, we have for any $c,c'$:
\[
2 d_G^{\mathcal{R}}(c,c') = 2 \sum_{k=0}^m d_G^{\mathcal{R}}(c^{(k)},c^{(k+1)}) \ge \sum_{k=0}^m d_G(c^{(k)},c^{(k+1)}) \ge d_G(c,c'), 
\]
where $m= d_G^{\mathcal{R}}(c,c')$ and $(c^{(k)})_{0 \le k \le m}$ are as defined in Notation~\ref{notation.distance.R}.
\end{proof}

\begin{lemma}\label{lemma.small.steps}
Let $c,c'$ be two non-backtracking cycles which differ by a square and $t$ a spine on $c_0$. Assume that $l(c) \ge l(c')$. Then
\[d^{\mathcal{R}}_G(c, t^{(l(c)-l(c'))/2} \odot c') \le \max \left(2, \frac{l(c)-l(c')}{2}\right).\]
\end{lemma}

\begin{proof}

We recall in Figure~\ref{figure.differ2} the three possible ways two non-backtracking cycles $c,c'$ can differ by a square. In particular we have $(l(c)-l(c')) \in 2 \mathbb{N}$.

\begin{figure}[h!]
\begin{center}
\begin{tikzpicture}[scale=0.3]
\draw (0,0) -- (1,1) -- (2,1) -- (3,2) -- (4,1) -- (5,1) -- (6,2);
\draw[-latex] (0,0) -- (0.5,0.5);
\draw[-latex] (5,1) -- (5.75,1.75);
\draw[-latex] (2,1) -- (2.75,1.75);
\draw[-latex] (2,1) -- (2.75,0.25);

\draw (2,1) -- (3,0) -- (4,1);
\draw[dashed] (6,2) -- (7,3);
\draw[dashed] (0,0) -- (-1,-1);
\node at (3,-1) {$c_n$};
\node at (3,3) {$c'_n$};
\node at (-2,1) {\textbf{(i)}};

\begin{scope}[xshift=12cm]
\draw (0.5,0) -- (1.5,1) -- (2.5,1) -- (4,1) -- (5,1) -- (6,2);
\draw[-latex] (0.5,0) -- (1,0.5);
\draw[-latex] (2.5,1) -- (3.5,1);
\draw[-latex] (2.5,-0.5) -- (3.5,-0.5);
\draw[-latex] (5,1) -- (5.75,1.75);
\draw (2.5,1) -- (2.5,-0.5) -- (4,-0.5) -- (4,1);
\draw[dashed] (6,2) -- (7,3);
\draw[dashed] (0.5,0) -- (-0.5,-1);
\node at (2.5,-1.5) {$c_n$};
\node at (4.75,-1.5) {$c_{n+1}$};
\node at (-2,1) {\textbf{(ii)}};
\end{scope}

\begin{scope}[xshift=24cm]
\draw (0,0) -- (1,1) -- (2,1) -- (4,1) -- (5,1) -- (6,2);
\draw[-latex] (0,0) -- (0.5,0.5);
\draw[-latex] (5,1) -- (5.75,1.75);
\begin{scope}[yshift=-3cm]
\draw (3,1) -- (2,0) -- (3,-1) -- (4,0) -- (3,1);
\node at (1.15,-.25) {$c_n$};
\node at (3.15,-1.75) {$c_{n+1}$};
\node at (5.5,-.25) {$c_{n+2}$};
\node at (3.15,5) {$c_{n-1-k} = c_{n+3+k}$};
\draw[-latex] (3,-1) -- (3.75,-0.25);
\end{scope}
\draw (3,1) -- (3,0.5);
\draw[dashed] (3,0.5) -- (3,-1.5);
\draw (3,-1.5) -- (3,-2);
\draw[dashed] (6,2) -- (7,3);
\draw[dashed] (0,0) -- (-1,-1);

\node at (-2,1) {\textbf{(iii)}};
\end{scope}
\end{tikzpicture}
\end{center}
\caption{The three ways two non-backtracking cycles can differ by a square. \label{figure.differ2}}
\end{figure}
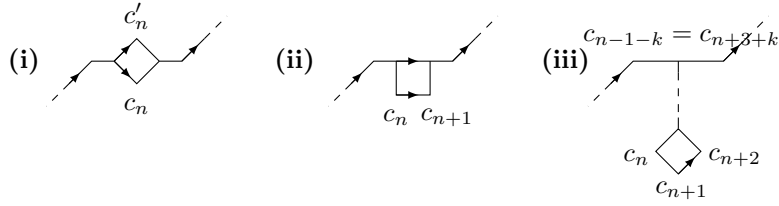

We consider each case one by one. In each of the cases $(l(c)-l(c')) \in \{0,2,4\}$ the table below provide a sequence of cycles from $c$ to $t^{l(c)-l(c')/2} \odot c'$ which yields the statement in this case. 

\begin{enumerate}[(i)]
\item $\boldsymbol{l(c)-l(c') = 0}$: \ \ \ \ 
\begin{tikzpicture}[scale=0.7, baseline=0.5em]
\draw[step=1,black,thin] (0,-1) grid (7,2);
\node[anchor=east] at (-.2, 1.5) {$c=$};

\node at (.5, 1.5) {$c_0$};
\node at (1.5, 1.5) {$\dots$};
\node at (2.5, 1.5) {$\dots$};
\node at (3.5, 1.5) {$c_{n-1}$};
\node at (4.5, 1.5) {$c_n$};
\node at (5.5, 1.5) {$c_{n+1}$};
\node at (6.5, 1.5) {$\dots$};

\node[anchor=east] at (-.2, .5) {$\rho_l(c)\ni$};

\node at (.5, 0.5) {$c_1$};
\node at (1.5, 0.5) {$\dots$};
\node at (2.5, 0.5) {$c_{n-1}$};
\node at (3.5, 0.5) {$c_n$};
\node at (4.5, 0.5) {$c_{n+1}$};
\node at (5.5, 0.5) {$\dots$};
\node at (6.5, 0.5) {$\dots$};

\node[anchor=east] at (-.2, -.5) {$c'=$};

\node at (.5, -0.5) {$c_0$};
\node at (1.5, -0.5) {$\dots$};
\node at (2.5, -0.5) {$\dots$};
\node at (3.5, -0.5) {$c_{n-1}$};
\node at (4.5, -0.5) {$c'_n$};
\node at (5.5, -0.5) {$c_{n+1}$};
\node at (6.5, -0.5) {$\dots$};
\end{tikzpicture}

\item $\boldsymbol{l(c)-l(c') = 2}$: \ \ \
\begin{tikzpicture}[scale=0.7, baseline=0.5em]
\draw[step=1,black,thin] (0,-1) grid (9,2);

\node[anchor=east] at (-.2, 1.5) {$c=$};
\node at (.5, 1.5) {$c_0$};
\node at (1.5, 1.5) {$\dots$};
\node at (2.5, 1.5) {$\dots$};
\node at (3.5, 1.5) {$c_{n-1}$};
\node at (4.5, 1.5) {$c_n$};
\node at (5.5, 1.5) {$c_{n+1}$};
\node at (6.5, 1.5) {$c_{n+2}$};
\node at (7.5, 1.5) {$\dots$};
\node at (8.5, 1.5) {$\dots$};

\node[anchor=east] at (-.2, 0.5) {$\rho_r(c) \ni$};
\node at (.5, 0.5) {$t_1$};
\node at (1.5, 0.5) {$c_0$};
\node at (2.5, 0.5) {$\dots$};
\node at (3.5, 0.5) {$\dots$};
\node at (4.5, 0.5) {$c_{n-1}$};
\node at (5.5, 0.5) {$c_{n}$};
\node at (6.5, 0.5) {$c_{n+1}$};
\node at (7.5, 0.5) {$c_{n+2}$};
\node at (8.5, 0.5) {$\dots$};

\node[anchor=east] at (-.2, -.5) {$t\odot c' =$};
\node at (.5, -0.5) {$t_0$};
\node at (1.5, -0.5) {$t_1$};
\node at (2.5, -0.5) {$c_0$};
\node at (3.5, -0.5) {$\dots$};
\node at (4.5, -0.5) {$\dots$};
\node at (5.5, -0.5) {$c_{n-1}$};
\node at (6.5, -0.5) {$c_{n+2}$};
\node at (7.5, -0.5) {$\dots$};
\node at (8.5, -0.5) {$\dots$};
\end{tikzpicture}

\item $\boldsymbol{l(c)-l(c') = 4}$: \ \
\begin{tikzpicture}[scale=0.7, baseline=-1.5em]
\draw[step=1,black,thin] (0,-3) grid (10,2);

\node[anchor=east] at (-.2, 1.5) {$c=$};
\node at (.5, 1.5) {$c_0$};
\node at (1.5, 1.5) {$\dots$};
\node at (2.5, 1.5) {$\dots$};
\node at (3.5, 1.5) {$c_{n-1}$};
\node at (4.5, 1.5) {$c_n$};
\node at (5.5, 1.5) {$c_{n+1}$};
\node at (6.5, 1.5) {$c_{n+2}$};
\node at (7.5, 1.5) {$c_{n+3}$};
\node at (8.5, 1.5) {$\dots$};
\node at (9.5, 1.5) {$\dots$};

\node[anchor=east] at (-.2, .5) {$\rho_r(c)\ni$};
\node at (.5, 0.5) {$t_1$};
\node at (1.5, 0.5) {$c_0$};
\node at (2.5, 0.5) {$\dots$};
\node at (3.5, 0.5) {$\dots$};
\node at (4.5, 0.5) {$c_{n-1}$};
\node at (5.5, 0.5) {$c_{n}$};
\node at (6.5, 0.5) {$c_{n+1}$};
\node at (7.5, 0.5) {$c_{n+2}$};
\node at (8.5, 0.5) {$c_{n+3}$};
\node at (9.5, 0.5) {$\dots$};

\node[anchor=east] at (-.2, -.5) {$\tilde c=$};
\node at (.5, -0.5) {$t_0$};
\node at (1.5, -0.5) {$t_1$};
\node at (2.5, -0.5) {$c_0$};
\node at (3.5, -0.5) {$\dots$};
\node at (4.5, -0.5) {$\dots$};
\node at (5.5, -0.5) {$c_{n-1}$};
\node at (6.5, -0.5) {$c_{n}$};
\node at (7.5, -0.5) {$c_{n+3}$};
\node at (8.5, -0.5) {$\dots$};
\node at (9.5, -0.5) {$\dots$};

\node[anchor=east] at (-.2, -1.5) {$\rho_r(\tilde c)\ni$};
\node at (.5, -1.5) {$t_1$};
\node at (1.5, -1.5) {$t_0$};
\node at (2.5, -1.5) {$t_1$};
\node at (3.5, -1.5) {$c_0$};
\node at (4.5, -1.5) {$\dots$};
\node at (5.5, -1.5) {$\dots$};
\node at (6.5, -1.5) {$c_{n-1}$};
\node at (7.5, -1.5) {$c_{n}$};
\node at (8.5, -1.5) {$c_{n+3}$};
\node at (9.5, -1.5) {$\dots$};

\node[anchor=east] at (-.2, -2.5) {$t^2\odot c'=$};
\node at (.5, -2.5) {$t_0$};
\node at (1.5, -2.5) {$t_1$};
\node at (2.5, -2.5) {$t_0$};
\node at (3.5, -2.5) {$t_1$};
\node at (4.5, -2.5) {$c_0$};
\node at (5.5, -2.5) {$\dots$};
\node at (6.5, -2.5) {$\dots$};
\node at (7.5, -2.5) {$c_{n-1}$};
\node at (8.5, -2.5) {$c_{n+4}$};
\node at (9.5, -2.5) {$\dots$};
\end{tikzpicture}
\end{enumerate}

When $l(c)-l(c') > 4$, let $k>0$ be such that for all $i \le k$, 
$c_{n+3+i}=c_{n-1-i}$ and $c_{n+3+k+1} \neq c_{n-1-k-1}$. By applying the case $l(c)-l(c')=4$, we have
\[d_G^{\mathcal{R}}(c,t^2 \odot c'')\le 2 \quad \text{where} \quad c''= c_0 \ldots c_{n-1} c_{n+4} \ldots c_{l(c)}.
\]
We thus only need to prove that $d_G^{\mathcal{R}}(t^2 \odot c'', t^{(l(c)-l(c'))/2} \odot c') \le k$. For this it is sufficient to see 
that
\[\forall i \in \llbracket 2, k+1\rrbracket,\ d_G^{\mathcal{R}}(t^{i} \odot c^{(i)}, t^{i+1} \odot c^{(i+1)})=1,\]
where we set 
\[c^{(i)} = c_0 \ldots c_{n-1-(i-2)} c_{n+4+(i-2)} \ldots c_{l(c)}.\qedhere\]
\end{proof}

By a repeated application of the previous lemma, we can transform a simple cycle into the power of a spine in bounded time:

\begin{lemma}\label{lemma.cycle.reduction}There exists a constant $\lambda_G$ such that, for any simple cycle $c$ which is decomposable into squares and spine $t$ on $c_0$, we have:
\[d^{\mathcal{R}}_G(t^{\lambda_G} \odot c, t^{\lambda_G +l(c)/2}) \le \lambda_G \qquad\text{and} \qquad l(c)\leq 2\lambda_G.\]
\end{lemma}

\begin{proof}
Recall that $m_c$ denotes the area of $c$ (see Definition~\ref{definition.decomposable}).
Denote by $\mathcal{D}_c$ the set of decompositions of length $m_c$ of a cycle $c$.
We define: 
\[\lambda_G := \max_{c \in \mathcal{C}_G^0} \min_{\left( c^{(i)}\right) \in \mathcal{D}_c} \left(\sum_{i \le m_c} \max \left(2, \left|\frac{l(c^{(i+1)})-l(c^{(i)})}{2}\right|\right)\right).\]
The definition should make clear that $l(c)\leq 2\lambda_G$ for all square-decomposable cycles $c$.
Let $(c^{(i)})_{0\le i\le m_c}$ be a square decomposition of $c$ that realizes the minimum in the definition of $\lambda_G$.
For all $j \le m_c$, let us set $\gamma^{(j)} : = t^{\lambda_G + (l(c) - l(c^{(j)}))/2} \odot c^{(j)}$. This sequence is well-defined because for all $j$, $\sum_{0\leq k< j}\left|\frac{l(c^{(k+1)})-l(c^{(k)})}{2} \right| \le \lambda_G$, and as a consequence:
\[\lambda_G + (l(c) - l(c^{(j)}))/2 \ge 0.\]
As a consequence of Lemma~\ref{lemma.small.steps}, for all $j \le m_c-1$, $m_c d^{\mathcal{R}}_G (\gamma^{(j)},\gamma^{(j+1)}) \le \lambda_G$. 
Thus by triangular inequality we get the statement of the lemma.
\end{proof}

\begin{corollary}\label{cor:reduction2m}
Let $c$ be a simple cycle which is decomposable into squares and let $t$ be a spine on $c_0$. We have:
\[
	d^{\mathcal{R}}_G(t^{\lambda_G} \odot c^{\lambda_G}, t^{\lambda_G + \lambda_G l(c)/2})\le \lambda_G ^2.
\]
\end{corollary}

\subsection{Parallelization\label{section.spine.parallelisation}}

In this section, we show how to perform the transformations defined in Section~\ref{section.additional.spines} in parallel on different parts of a cycle\textit{}. 

We introduce (Notation~\ref{notation.gamma}) $\Gamma^{\textbf{r}}_{\textbf{z}} (d)$ which is a cycle obtained by ``inserting'' a sequence of cycles $\textbf{r}$ inside the cycle $d$ at positions given by $\textbf{z}$, in a similar way as in the definition of a cactus. This lets us consider transformations on different parts of $d$ while leaving unchanged the inserted cycles of the sequence $\textbf{r}$.

The main results of this section are Lemma~\ref{lem:lognm} and Lemma~\ref{proposition.log.upper.bound.general}, which prove respectively that the words of $\textbf{r}$ in  $\Gamma^{\textbf{r}}_{\textbf{z}}(d)$ can be moved in parallel; and that transformations such as the ones of Section~\ref{section.additional.spines} can be performed in parallel on the different parts of $d$. 

\begin{notation}\label{notation.gamma}
Let $\textbf{z} = (\textbf{z}_0, \dots , \textbf{z}_l)$ be a sequence of non-negative integers and $\textbf{r}=(r^{(j)})_{0 \le j \le l}$ be a sequence of cycles all beginning and ending at some vertex $a\in V_G$. For any cycle $d$ of length $\sum_{j=0}^l \textbf{z}_j$ such that for all $j \in \llbracket 0,l\rrbracket$, $d_{\textbf{z}_0 + \ldots + \textbf{z}_j} = d_0 = a$, we define:
\[
\Gamma^{\textbf{r}}_{\textbf{z}}(d) := d_{|\llbracket 0, \textbf{z}_0 \rrbracket} \odot r^{(0)} \odot d_{|\textbf{z}_0 + \llbracket 0, \textbf{z}_1\rrbracket} \odot \hdots \odot d_{| \textbf{z}_0 + \ldots +\textbf{z}_{l-1}+\llbracket 0 , \textbf{z}_{l}\rrbracket} \odot r^{(l)}.
\]
\end{notation}

\begin{lemma} \label{lemma.parallelisation}
For: 
\begin{itemize}
\item any sequence $\textbf{r}=(r^{(j)})_{0 \le j \le l}$ of cycles all beginning and ending at some vertex $a\in V_G$ and any sequence $\textbf{z}=(z^{(j)})_{0 \le j \le l}$ of integers, 
\item any sequence $(\epsilon_k)_{0 \le k \le m-1}$ in $\{0,1\}^{m}$,
\item any double sequence of cycles $(c^{(k,i)})_{(k,i) \in \llbracket 0,m \rrbracket \times \llbracket 0,l\rrbracket}$ which all begin and end at $a$, such that $l\left(c^{(k,i)}\right)=z_i$ and $c^{(k+1,i)} \mathcal{R}_{\epsilon_k} c^{(k,i)}$ for all $(k,i)$,
\end{itemize}
we have:
\[d^{\mathcal{R}}_G(\Gamma^{\textbf{r}}_{\textbf{z}}(c^{(0,0)} \odot \dots \odot c^{(0,l)}), \Gamma^{\textbf{r}}_{\textbf{z}}(c^{(m,0)} \odot \dots \odot c^{(m,l)})) \le m.\] 
\end{lemma} 

\begin{proof}
It is sufficient to see for all $k<m$, we have 
\[\Gamma^{\textbf{r}}_{\textbf{z}}(c^{(k+1,0)} \odot \dots \odot c^{(k+1,l)}) \mathcal{R}_{\epsilon_k} \Gamma^{\textbf{r}}_{\textbf{z}}(c^{(k,0)} \odot \dots \odot c^{(k,l)}).\]
This derives from the fact that $r^{(i)} \mathcal{R}_{\epsilon} r^{(i)}$ for all $i$ and $\epsilon$, and that the second letter and the penultimate letter of $r^{(i)}$ are both neighbors of $a$. 
\end{proof}

\begin{remark}\label{lemma.parallelisation.second}
If we have a sequence of cycles \[
(c^{(i,j)})_{\substack{1 \le i \le k\\0 \le j \le l_i}}
\]
such that $c^{(i,j)} \mathcal{R}_0 c^{(i,j+1)}$  for all $i$ and $j < l_i$, we can complete this sequence into 

\[
(c^{(i,j)})_{\substack{1 \le i \le k\\0 \le j \le L}},
\]
where $L := \max_i l_i$, by setting $c^{(i,j)} := c^{(i,l_i)}$
when $j > l_i$, and then apply Lemma~\ref{lemma.parallelisation}
in order to get: 
\[d^{\mathcal{R}}_G(\Gamma^{\textbf{r}}_{\textbf{z}}(c^{(1,0)}\odot c^{(2,0)} \odot \hdots \odot c^{(k,0)}) , \Gamma^{\textbf{r}}_{\textbf{z}}(c^{(1,L)}\odot c^{(2,L)} \odot \hdots \odot c^{(k,L)}) ) \le L,\]
where $\textbf{z} = (l(c^{(i)}))_{1\leq i\leq k}$.
\end{remark}

\begin{lemma}\label{lem:changing-sequence}
Let us fix $\textbf{r}=(r^{(j)})_{0 \le j \le l}$ a sequence of cycles all beginning and ending at some vertex $a\in V_G$.
Let us also consider a pair $(\textbf{z}$, $\textbf{z}')$ for which there exists an increasing sequence $(j_i)_{1 \le i \le 2\tau}$ of integers in $\llbracket 0, l \rrbracket$ such that: 
\begin{itemize}
\item $\forall i \in \llbracket 1, 2t\rrbracket$, $\textbf{z}_{j_i} = \textbf{z}'_{j_i} + 2\cdot(-1)^{i}$;
\item for all $j$ which are not in the sequence $(j_i)_{1 \le i \le 2\tau}$, $\textbf{z}'_{j} = \textbf{z}_{j}$.
\end{itemize}
Then, for all cycles $d$ such that $\Gamma^{\textbf{r}}_{\textbf{z}}(d)$ and $\Gamma^{\textbf{r}}_{\textbf{z}'}(d)$ are well defined, we have 
\[
d_G^{\mathcal{R}}(\Gamma^{\textbf{r}}_{\textbf{z}}(d), \Gamma^{\textbf{r}}_{\textbf{z}'}(d))=1.
\]
\end{lemma}
\begin{proof}
That $\Gamma^{\textbf{r}}_{\textbf{z}}(d)$ and $\Gamma^{\textbf{r}}_{\textbf{z}'}(d)$ are well-defined means that for all $j$,
\[d_{|\textbf{z}_0 + \ldots +\textbf{z}_{j}} = d_{|\textbf{z}'_0 + \ldots +\textbf{z}'_{j}} = a.\]
For any $j$, as a consequence of the hypothesis, the difference $\sum_{i \le j}\textbf{z}'_i - \sum_{i \le j}\textbf{z}_i$ is equal to $0$ or $2$. Whenever $\sum_{i \le j}\textbf{z}'_i - \sum_{i \le j}\textbf{z}_i = 2$,
since 
\[d_{|\textbf{z}_0 + \ldots +\textbf{z}_{j}} = d_{|\textbf{z}'_0 + \ldots +\textbf{z}'_{j}} = a,\]
this means that $d_{|\textbf{z}'_0 + \ldots +\textbf{z}'_{j}+\llbracket-2, 0\rrbracket}$ is a spine on $a$ that we denote by $t$. Therefore $d$ can be decomposed as follows: 
\[d = \left(d^{(0)} \odot t^{(0)}\right) \odot \ldots \odot \left(d^{(l)} \odot t^{(l)}\right),\]
where $t^{(j)}$ is a spine on $a$ when $\sum_{i<j}\textbf{z}'_i - \sum_{i<j}\textbf{z}_i = 2$ and $t^{(j)} =a$ (empty cycle) otherwise.
This is illustrated on Figure~\ref{fig:decomp}.
\begin{figure}[h!]
\begin{tikzpicture}
\draw (0,0.2) rectangle (12,0.8);
\node at (1.5,0.5) {$d^{(0)}$};
\node at (4.5,0.5) {$d^{(1)}$};
\node at (7.5,0.5) {$d^{(2)}$};
\node at (10.5,0.5) {$d^{(3)}$};

\draw (2.7, 0) -- ++(0,1);
\node at (3, 0.5) {$t^{(0)}$};
\draw (3.3, 0) -- ++(0,1);

\draw (6, 0) -- ++(0,1);

\draw (8.7, 0) -- ++(0,1);
\node at (9, 0.5) {$t^{(2)}$};
\draw (9.3, 0) -- ++(0,1);

\draw [<->] (0, -0.2) -- node[below]{$\textbf{z}_0$} (2.65, -0.2);
\draw [<->] (2.75, -0.2) -- node[below]{$\textbf{z}_1$} (5.95, -0.2);
\draw [<->] (6.05, -0.2) -- node[below]{$\textbf{z}_2$} (8.65, -0.2);
\draw [<->] (8.75, -0.2) -- node[below]{$\textbf{z}_3$} (12, -0.2);

\draw [<->] (0, 1.2) -- node[above]{$\textbf{z}'_0$} (3.25, 1.2);
\draw [<->] (3.35, 1.2) -- node[above]{$\textbf{z}'_1$} (5.95, 1.2);
\draw [<->] (6.05, 1.2) -- node[above]{$\textbf{z}'_2$} (9.25, 1.2);
\draw [<->] (9.35, 1.2) -- node[above]{$\textbf{z}'_3$} (12, 1.2);
\end{tikzpicture}
\caption{Decomposition of $d$ in the case where the sequence $(j_i)$ is $(0,1,2,3)$.}\label{fig:decomp}
\end{figure}
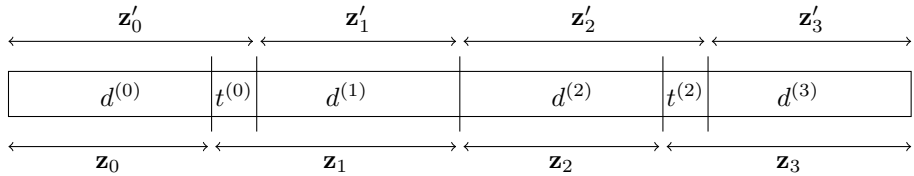

Furthermore, 

\[\Gamma^{\textbf{r}}_{\textbf{z}}(d) =  \left( d^{(0)} \odot \left( t^{(0)} \odot r^{(0)} \right) \right) \odot  \ldots \odot \left(d^{(l)} \odot \left( t^{(l)} \odot r^{(l)} \right)\right)\]

and

\[\Gamma^{\textbf{r}}_{\textbf{z}'}(d) =  \left( d^{(0)} \odot \left( r^{(0)} \odot t^{(0)} \right) \right) \odot  \ldots \odot \left(d^{(l)} \odot \left( r^{(l)} \odot t^{(l)} \right)\right)\]

Now the fact that $t\odot r^{(i)} \mathcal{R}_0 r^{(i)}\odot t$ for all $i$, together with Lemma~\ref{lemma.parallelisation}, implies the result.
\end{proof}


\begin{lemma}\label{lem:lognm}
Let us fix $\textbf{r}=(r^{(j)})_{0 \le j \le l}$ a sequence of cycles all beginning and ending at some vertex $a\in V_G$.
Let $c$ be a simple cycle decomposable into squares and $t$ a spine on $c_0$.
For all $n \ge0$, we have
\[
d_G^\mathcal{R}(\Gamma^{\textbf{r}}_\textbf{z}(t^{\lambda_G l(c)/2}\odot c^{(2^n-1)\lambda_G}),\ \Gamma^{\textbf{r}}_\textbf{z}(t^{2^{n-1} \lambda_G l(c)})) \le 30 \lambda_G^2 n
\]
where $\textbf{z} := (\lambda_G l(c), k_1 l(c), \ldots , k_l l(c))$ for any sequence of positive integers $(k_j)_{1 \le j \le l}$ such that $k_1 + \hdots + k_l = (2^n-1)\lambda_G$.
\end{lemma}

\paragraph{Sketch of the proof:} \textit{We transform a cycle of the form $t^{k}\odot c^{2^n}$ into a trivial cycle (containing only spines). To do this, we move enough copies of the spine $t$ from the left to the center of the word that we are able to apply results of Section~\ref{section.additional.spines} at the center of the word. This transforms a number of occurrences of $c$ at the center into spines; after this the word  consists in two identical blocks of the form $t^k  \odot c^{2^{n-1}}$. We repeat the previous transformation on each of these two blocks, etc. This `dichotomic' process finishes in time which is linear in $n$.}

\begin{proof}
When $l(c)=2$, there is nothing to prove. Let us assume that $l(c) \ge 4$.

\begin{enumerate}
\item \textbf{The dichotomic process:}
For all $k \le n$, let us set:
\[
\gamma_{k} := \left(t^{\lambda_G l(c)/2} \odot c^{\left(2^{n-k}-1\right) \lambda_G}\right)^{2^k},
\]
where $\lambda_G$ is as defined in Lemma~\ref{lemma.cycle.reduction}.
In particular $\gamma_0 = t^{l(c)\lambda_G /2}\odot c^{(2^n-1)\lambda_G}$ and $\gamma_n = t^{2^{n-1} \lambda_G l(c)}$.

Let us prove that there are sequences $\textbf{z}^k$, $k \le n$, such that $\textbf{z}^0 = \textbf{z}$ and
\[
\forall k<n,\ d_G^\mathcal{R}(\Gamma^{\textbf{r}}_{\textbf{z}^k}(\gamma_{k}),\Gamma^{\textbf{r}}_{\textbf{z}^{k+1}}(\gamma_{k+1})) \le { 15 \lambda_G ^2}.
\]
\item \textbf{From $\gamma_{k-1}$ to $\gamma_k$ - moving spines:} Let us assume that $\textbf{z}^0, \dots , \textbf{z}^{k-1}$ have been defined, where $k \ge 1$. Let the sequence of cycles $(\gamma_{k,j})_{j \in\llbracket 0 ,2\lambda_G \rrbracket}$ be defined for all $k$ and $j$ as the concatenation of $2^{k-1}$ identical words that we call \textbf{blocks}:
\[
\gamma_{k,j} =\left(t^{\lambda_G l(c)/2-j} \odot c^{\left(2^{n-k}-1\right)\lambda_G} \odot t^j \odot c^{2^{n-k}\lambda_G} \right)^{2^{k-1}}.
\]
Informally, for all $j$, $\gamma_{k,j+1}$ is obtained from  $\gamma_{k,j}$ by `moving' a spine $t$ from the left to the right inside each block. This will also affect the sequence $\textbf{z}^{k,j}$ responsible for places of insertions of $\mathbf{r}$. Then by Lemma~\ref{lemma.parallelisation},
\[d^{\mathcal{R}}_G\left(\Gamma^{\textbf{r}}_{\textbf{z}^{k,j}}(\gamma_{k,j}), \Gamma^{\textbf{r}}_{\textbf{z}^{k,j}}(\gamma_{k,j+1})\right) = 1.\]
Let us construct a sequence $\left(\textbf{z}^{k,j}\right)_{j \in\llbracket 0 ,2\lambda_G \rrbracket}$, such that $\textbf{z}^{k,0} = \textbf{z}^{k-1}$. Let us assume that this sequence has been defined for some $j < 2\lambda_G$. 

In the definition of $\Gamma^{\textbf{r}}_{\textbf{z}^{k,j}}(\gamma_{k,j})$ it may happen that $\Gamma^\textbf{r}$ inserts an element of $\textbf{r}$ in between $t^{\lambda_G l(c)/2}$ and $t^j$, which changes the relative position of the spine being moved in the $i$-th block when constructing $\gamma_{k,j+1}$ from $\gamma_{k,j}$. To cancel out this movement, we define a quantity $\omega^{k,j+1}_q$ to be $1$ when this happens, and $0$ otherwise, and we define:
\begin{itemize}
    \item $\textbf{z}^{k,j+1}_0 = \textbf{z}^{k,j}_0  - 2 \omega^{k,j+1}_0$;
    \item for all $q>0$, $\textbf{z}^{k,j+1}_q = \textbf{z}^{k,j}_q  - 2 \omega^{k,j+1}_q + 2 \omega^{k,j+1}_{q-1}$.
\end{itemize}
Formally, denote $(\omega^{k,j+1}_q)_{0 \le q \le 
l} \in \{0,1\}^{l+1}$  such that for all 
$q$, 
$\omega^{k,j+1}_q = 1$ when there exists some $i \in \llbracket 0, 2^{k-1}-1\rrbracket$ for which 
\[\underbrace{\left(\lambda_G l(c) - 2j\right) + \overbrace{i \cdot \lambda_G 2^{n-k+1} l(c)}^{\text{the spine is in block $i+1$}}}_{\text{position of spine $t$ before moving}}  < \underbrace{\textbf{z}^{k,j}_0 + \ldots + \textbf{z}^{k,j}_q}_{\text{insertion position of $r^{(q)}$}}\]
and 
\[ \textbf{z}^{k,j}_0 + \ldots + \textbf{z}^{k,j}_q \leq \underbrace{\left(2^{n-k} \lambda_G l(c) - 2j\right)+ i \cdot \lambda_G 2^{n-k+1} l(c)}_{\text{position of spine $t$ after moving}},\]

and $\omega^{k,j+1}_q=0$ otherwise.  

From this definition, we have that
\begin{itemize}
    \item $\textbf{z}^{k,j+1}_q = \textbf{z}^{k,j}_q$ whenever $\omega^{k,j+1}_q = \omega^{k,j+1}_{q-1}$ ;
    \item $\textbf{z}^{k,j+1}_q = \textbf{z}^{k,j}_q + 2$ whenever $\omega^{k,j+1}_q = 0$ and $\omega^{k,j+1}_{q-1} = 1$ ;
    \item $\textbf{z}^{k,j+1}_q = \textbf{z}^{k,j}_q - 2$ whenever $\omega^{k,j+1}_q = 1$ and $\omega^{k,j+1}_{q-1} = 0$.
\end{itemize}
Necessarily these two last types of indexes alternate. This implies that it is 
possible to apply Lemma~\ref{lem:changing-sequence} in which $\textbf{z},\textbf{z}'$ are replaced respectively by $\textbf{z}^{k,j+1}$ and $\textbf{z}^{k,j}$.

By Lemma~\ref{lem:changing-sequence}, we have:
\[d^{\mathcal{R}}_G\left(\Gamma^{\textbf{r}}_{\textbf{z}^{k,j}}(\gamma_{k,j+1}), \Gamma^{\textbf{r}}_{\textbf{z}^{k, j+1}}(\gamma_{k,{j+1}})\right) = 1.
\]
We have $\gamma_{k,0} = \gamma_{k-1}$, $\textbf{z}^{k,0} = \textbf{z}^{k-1}$ and $\gamma_{k,2\lambda_G} = \gamma'_{k}$ where:
\[
\gamma'_{k} := \left(t^{(l(c)/2-2)\lambda_G} \odot c^{\left(2^{n-k}-1\right)\lambda_G} \odot t^{2\lambda_G} \odot c^{2^{n-k}\lambda_G} \right)^{2^{k-1}}.
\]
Thus by triangular inequality,
\begin{equation}\label{eq:1}
d^{\mathcal{R}}_G \left( \Gamma^{\textbf{r}}_{\textbf{z}^{k-1}}(\gamma_{k-1}), \Gamma^{\textbf{r}}_{\textbf{z}^{k, 2\lambda_G}}(\gamma'_{k})\right)\le 4\lambda_G.
\end{equation}

In $\Gamma^{\textbf{r}}_{\textbf{z}^{k, 2\lambda_G}}(\gamma'_{k})$, it may happen that the cycles $r^{(j)}$ appear inside an occurrence of $t^{2\lambda_G} \odot c$.
We change the sequence $\textbf{z}^{k, 2\lambda_G}$ by moving its elements to the left when it happens, in a similar way as we did with $\omega_q^{k,j+1}$ above, so that we can do further modifications to the block in question.

Formally, we define a sequence $\textbf{x}^{0}, \dots , \textbf{x}^{m}$ as follows. Put $\textbf{x}^0 = \textbf{z}^{k, 2\lambda_G}$. For all $p$, consider the following cases in order to define $\textbf{x}^p$:
\begin{itemize}
    \item for some indices $q$, $\textbf{x}^p_0 + \dots + \textbf{x}^p_q$ falls in $\Gamma^{\textbf{r}}_{\textbf{x}^{p}}(\gamma'_{k})$ inside an occurrence of  $t^{2\lambda_G} \odot c$. Then we set:
    \begin{itemize}
        \item $\textbf{x}^{p+1}_q = \textbf{x}^{p}_q - 2$ if $q$ is such an index but $q-1$ is not (or $q=0$);
        \item $\textbf{x}^{p+1}_q = \textbf{x}^{p}_q+2$ if $q$ is not such an index but $q-1$ is;
        \item $\textbf{x}^{p+1}_q = \textbf{x}^{p}_q$ in all other cases.
    \end{itemize}
    \item If there are no such indices, the sequence ends at $\textbf{x}^{p}$ and we set $m=p$.
\end{itemize}

We can check that $m \le 3\lambda_G$.

In the end, we have that every occurrence of the word $t^{2\lambda_G} \odot c$ in $\gamma'_{k+1}$ is in some sub-word $\left(\gamma'_{k+1}\right)_{\textbf{z}^m_0 + \ldots \textbf{z}^m_{j} + \llbracket 0, \textbf{z}^m_{j+1} \rrbracket}$ for some $j < l$. By another application of Lemma~\ref{lem:changing-sequence}, we have:
\begin{equation}\label{eq:2}
	d^{\mathcal{R}}_G \left(\Gamma^{\textbf{r}}_{\textbf{z}^{k,2\lambda_G}} (\gamma'_{k}), \Gamma^{\textbf{r}}_{\textbf{x}^m} (\gamma'_{k})\right) \le {3\lambda_G}.
\end{equation}
\item \textbf{Dichotomic decomposition:}
Similarly as in Corollary~\ref{cor:reduction2m}, we decompose in each block the copies of $c$ in $c^{\lambda_G}$ using the spines $t^{2\lambda_G}$. For all $m$, we set: 
\[\gamma'_{k,m} := \left(t^{(l(c)/2-2)\lambda_G} \odot c^{\left(2^{n-k}-1\right)\lambda_G} \odot t^{2\lambda_G + ml(c)/2} \odot c^{\left(2^{n-k}\right)\lambda_G-m} \right)^{2^{k-1}}.
\]
We have $\gamma'_{k,0} = \gamma'_k$ and, by Lemma~\ref{lemma.cycle.reduction} and Lemma~\ref{lemma.parallelisation},
\[
d^{\mathcal{R}}_G\left(\Gamma^{\textbf{r}}_{\textbf{y}^0} (\gamma'_{k,0}), \Gamma^{\textbf{r}}_{\textbf{y}^0} (\gamma'_{k,1}) \right) \le 2\lambda_G,
\] 
where we set $\textbf{y}^0 := \textbf{x}^m$.

Some cycles in $\textbf{r}$ may appear inside an occurrence of $t^{2\lambda_G} \odot c$ in $\Gamma^{\textbf{r}}_{\textbf{x}^0} (\gamma'_{k,1})$. This is the reason why, instead of just applying 
 Corollary~\ref{cor:reduction2m}, we define a sequence 
 $\gamma'_{k,m}$, $0 \le m \le \lambda_G$ step by step: we need to correct 
 each time these `misplacements'. Applying the same argument as in the previous step, there is some $\textbf{y}^1$ such that this is no longer the case and 
\[
	d^{\mathcal{R}}_G \left(\Gamma^{\textbf{r}}_{\textbf{y}^0} (\gamma'_{k,1}), \Gamma^{\textbf{r}}_{\textbf{y}^1} (\gamma'_{k,1}) \right) \le { 3\lambda_G}.
\]
In a similar manner, we build a sequence $\left(\textbf{y}^m\right)_{1 \le m \le \lambda_G}$ such that for all $m$, 
\[
d^{\mathcal{R}}_G\left(\Gamma^{\textbf{r}}_{\textbf{y}^{m-1}} (\gamma'_{k,m-1}), \Gamma^{\textbf{r}}_{\textbf{y}^m} (\gamma'_{k,m}) \right) \le 2\lambda_G + { 3 \lambda_G},
\]
Summing up,
\begin{equation}\label{eq:3}
d^{\mathcal{R}}_G\left(\Gamma^{\textbf{r}}_{\textbf{y}^{0}} (\gamma'_{k,0}), \Gamma^{\textbf{r}}_{\textbf{y}^{\lambda_G}} (\gamma'_{k,\lambda_G}) \right) \le { 5}\lambda_G^2.
\end{equation}
Now notice that $\gamma'_{k,\lambda_G}$ is the same word as $\gamma_k$, except for some blocks of $2\lambda_G$ `misplaced' spines. We can move them back in time ${3\lambda_G}$, exactly as in Step 2 of the current proof: thus there is some $\textbf{z}^k$ such that 
\begin{equation}\label{eq:4}
d^{\mathcal{R}}_G\left(\Gamma^{\textbf{r}}_{\textbf{y}^{\lambda_G}} (\gamma'_{k,\lambda_G}), \Gamma^{\textbf{r}}_{\textbf{z}^k} (\gamma_k) \right) \le { 3\lambda_G}.
\end{equation}

 Summing Equations~\ref{eq:1}, \ref{eq:2}, \ref{eq:3} and \ref{eq:4}, and by triangular inequality:
\[
d^{\mathcal{R}}_G(\Gamma^{\textbf{r}}_{\textbf{z}^{k-1}}(\gamma_{k-1}),\Gamma^{\textbf{r}}_{\textbf{z}^k}(\gamma_{k})) \le 4\lambda_G + { 3\lambda_G} + { 5\lambda_G ^2} + { 3\lambda_G} \le { 15\lambda_G^2}.
\]

\item \textbf{Conclusion:} As a consequence, again by triangular inequality:
\[
d_G^\mathcal{R}(\Gamma^{\textbf{r}}_{\textbf{z}}(t^{\lambda_G l(c)/2}\odot c^{(2^n-1)\lambda_G}),\ \Gamma^{\textbf{r}}_{\textbf{z}^{n}}(t^{2^{n-1} \lambda_G l(c)})) \le { 15 \lambda_G^2 n}
\]

The previous equation implies that, to go from the sequence $\textbf{z}$ to $\textbf{z}^n$, we had to apply Lemma~\ref{lem:changing-sequence} at most $15 \lambda_G^2 n$ times. Each time, we apply the same lemma to come back to the initial sequence, which leads to the following estimate:
\[
d_G^\mathcal{R}(\Gamma^{\textbf{r}}_{\textbf{z}^{n}}(t^{2^{n-1}\lambda_G l(c)})), \Gamma^{\textbf{r}}_{\textbf{z}}(t^{2^{n-1} \lambda_G l(c)}))) \le { 15 \lambda_G ^2 n}.
\]

By triangular inequality, we obtain: 
\[
d_G^\mathcal{R}(\Gamma^{\textbf{r}}_{\textbf{z}}(t^{\lambda_G l(c)/2}\odot c^{(2^n-1)\lambda_G}),\ \Gamma^{\textbf{r}}_{\textbf{z}}(t^{2^{n-1} \lambda_G l(c)})) \le { 30 \lambda_G ^2} n,
\]
which concludes the proof.\qedhere
	\end{enumerate}
\end{proof}

\begin{lemma}\label{proposition.log.upper.bound.general}
If $G$ is square decomposable then there is a constant $\alpha_G>0$ such that, for every simple cycle $c$, spine $t$ on $c_0$, integer $n$, and sequence of cycles $\textbf{r}=(r^{(j)})_{0 \le j \le l}$ all beginning and ending at some vertex $a\in V_G$, we have:
\[d_G (\Gamma^{\textbf{r}}_\textbf{z}(t^{\lambda_G l(c)/2} \odot c^{n}), \Gamma^{\textbf{r}}_\textbf{z}(t^{(\lambda_G +n)l(c)/2}))\le \alpha_G \cdot \log_2(n),\]
where $\textbf{z} := (\lambda_G l(c), k_1 l(c), \ldots , k_l l(c))$ for any sequence of positive integers $(k_j)_{1 \le j \le l}$ such that $k_1 + \hdots + k_l = n$.
\end{lemma}

\begin{proof}
Let us fix a cycle $c$, and some integer $n\ge \lambda_G$. Let us denote 
by $k \ge 1$ and $s<\lambda_G$ integers such that $n=k\lambda_G-s$.

By Lemma~\ref{lem:lognm}, 
\[
d_G^\mathcal{R}\left(\Gamma^{\textbf{r}}_\textbf{z}\left(t^{\lambda_G l(c)/2} \odot c^{\lambda_G (2^{\lceil \log_2(k) \rceil+1}-1)}\right), \Gamma^{\textbf{r}}_\textbf{z}\left(t^{\lambda_G l(c)2^{\lceil \log_2(k) \rceil}}\right)\right) \le { 30} \lambda_G ^2 (\log_2(k) + 2),
\]
and by Lemma~\ref{lem:comp-distances},
\[
d_G \left(\Gamma^{\textbf{r}}_\textbf{z}\left(t^{\lambda_G l(c)/2} \odot c^{\lambda_G (2^{\lceil \log_2(k) \rceil+1}-1)}\right), \Gamma^{\textbf{r}}_\textbf{z}\left(t^{\lambda_G l(c)2^{\lceil \log_2(k) \rceil}}\right)\right) \le { 60} \lambda_G ^2 (\log_2(k) + 2).
\]
It is straightforward that $d_G$ is conserved when taking prefixes of the same length. Since $n\ge \lambda_G$, we have $\lambda_G (2^{\lceil \log_2(k) \rceil+1}-1) \geq 2\lambda_G k - \lambda_G \geq 2(n+s) - \lambda_G \geq n$. It follows that
$\Gamma^{\textbf{r}}_\textbf{z}(t^{\lambda_G l(c)/2} \odot c^{n})$
is a prefix of
$\Gamma^{\textbf{r}}_\textbf{z}(t^{\lambda_G l(c)/2} \odot c^{\lambda_G (2^{\lceil \log_2(k) \rceil+1}-1)})$. Therefore:
\[d_G(\Gamma^{\textbf{r}}_\textbf{z}(t^{\lambda_G l(c)/2} \odot c^{n}), \Gamma^{\textbf{r}}_\textbf{z}(t^{(\lambda_G +n)l(c)/2}))\le {60} \lambda_G ^2 \cdot (\log_2(k)+2) \le 240 \lambda_G ^2 \cdot \log_2 (n).\]
This inequality is still satisfied when replacing $240 \lambda_G ^2$ by any value $\alpha_G \geq 240 \lambda_G ^2$. We can thus take $\alpha_G$ sufficiently large that the inequality holds for all simple cycles $c$ and all $n < \lambda_G$ as well.
\end{proof}

\subsection{The gap function is logarithmic\label{section.main.finite.cover}}

In this section, we prove Theorem~\ref{theorem.square.dismantable.to.log}. By representing a cycle as a cactus, we can apply the transformation defined in the previous sections in parallel on all leaves. Doing this in a repeated way, every cycle can be transformed into a power of a spine in a time which depends logarithmically on the length of the cycle. 

Before this, we need a last technical tool. Let us recall that we have, for any cycle $c$, any spine $t$ on $c_0$, and any $m \ge 0$: 
\[
d^{\mathcal{R}}_G(t^m \odot c,c \odot t^m) \le m.
\]
The next lemma provides another useful bound that holds in a more general context.

\begin{lemma}\label{lemma:movingbacktracks}
Let $u$ be a walk and let $t, t'$ be spines on $u_0$ and $u_{l(u)}$, respectively. Then
\[
d^{\mathcal{R}}_G(t^m\odot u, u\odot t'^m) \le l(u).
\]
\end{lemma}

\begin{proof}
Let $t^{(i)}$ be the spine $u_iu_{i+1}u_i$ and $u^{(i)} = u_0\dots u_i \odot \left(t^{(i)}\right)^m \odot u_i\dots u_{l(u)}$. It is enough to check that $(t^m \odot u) \mathcal{R}_0 u^{(1)}$, that $u^{(i)} \mathcal{R}_0 u^{(i+1)}$ for all $0<i<l(u)-1$, and that $u^{(l(u)-1)} \mathcal{R}_0 (u \odot t'^m)$.
\end{proof}

\begin{remark}
Lemma~\ref{lemma:movingbacktracks} still holds when, instead of $u\odot t'^m$, spines appear in $u$ in arbitrary positions. This allows us to gather all spines on the left or right side of $u$ within $l(u)$ bound on $d^{\mathcal{R}}_G$.
\end{remark}

\begin{lemma}\label{lemma:extension}
Take $a,b$ in $V_G$, $r$ and $r'$ simple walks respectively from $b$ to $a$ and from $a$ to $b$, and $c, c'$ two cycles of same length from $a$ to $a$. Then:
\[d_G(r\odot c \odot r',r\odot c' \odot r') \le d_G(c,c') + 4 \text{diam}(G).\]
\end{lemma} 

\begin{proof}
Let $\left(p^{(i)}\right)_{0\le i \le m}$ be a sequence of walks on $G$, where $m=d_G(c,c')$, which all have the same 
length $l(c)$ such that for all $i$, $p^{(i+1)}$ and $p^{(i)}$ are neighbors 
and $p^{(0)}=c,p^{(m)}=c'$. Let us define two other sequences $r_l^{(i)}$ and 
$r_r^{(i)}$ such that $r_l^{(0)} = r, r_r^{(0)}=r'$, and for all $i$, 
$r_l^{(i+1)}$ (resp. $r_r^{(i+1)}$) is the left (resp. right) shift of $r_l^{(i)}$ which ends (resp. begins) on $p^{(i+1)}_0$ (resp. $p^{(i+1)}_{l(c)}$). Then the sequence of walks 
\[q^{(i)} := r_l^{(i)} \odot p^{(i)} \odot r_r^{(i)}\]
is well-defined and for all $i$, $q^{(i+1)}$ and $q^{(i)}$ are neighbors.
Since $r$ and $r'$ have length smaller than $\text{diam}(G)$, 
by shifting $\text{diam}(G)$ times to the left, $\text{diam}(G)$ times to the right, $\text{diam}(G)$ times to the right and then $\text{diam}(G)$ times to the left again, we can find a walk on $\Delta_G^{l(c)}$ from $q^{(m)}$ to $r\odot c' \odot r'$. Triangular inequality implies the statement of the lemma.
\end{proof}

\begin{theorem}\label{theorem.square.dismantable.to.log}
If $G$ is square decomposable, $\gamma_G(n) = O(\log(n))$.
\end{theorem}

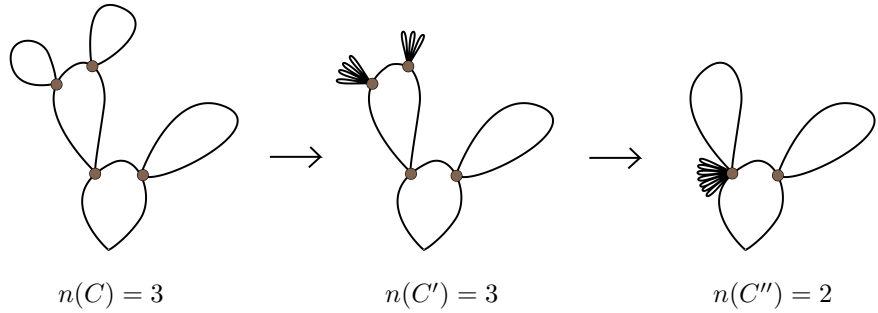
\begin{figure}[h!]
	\begin{center}
	
	\begin{tikzpicture}
	\node at (0,0) {\includegraphics[width=1\textwidth]{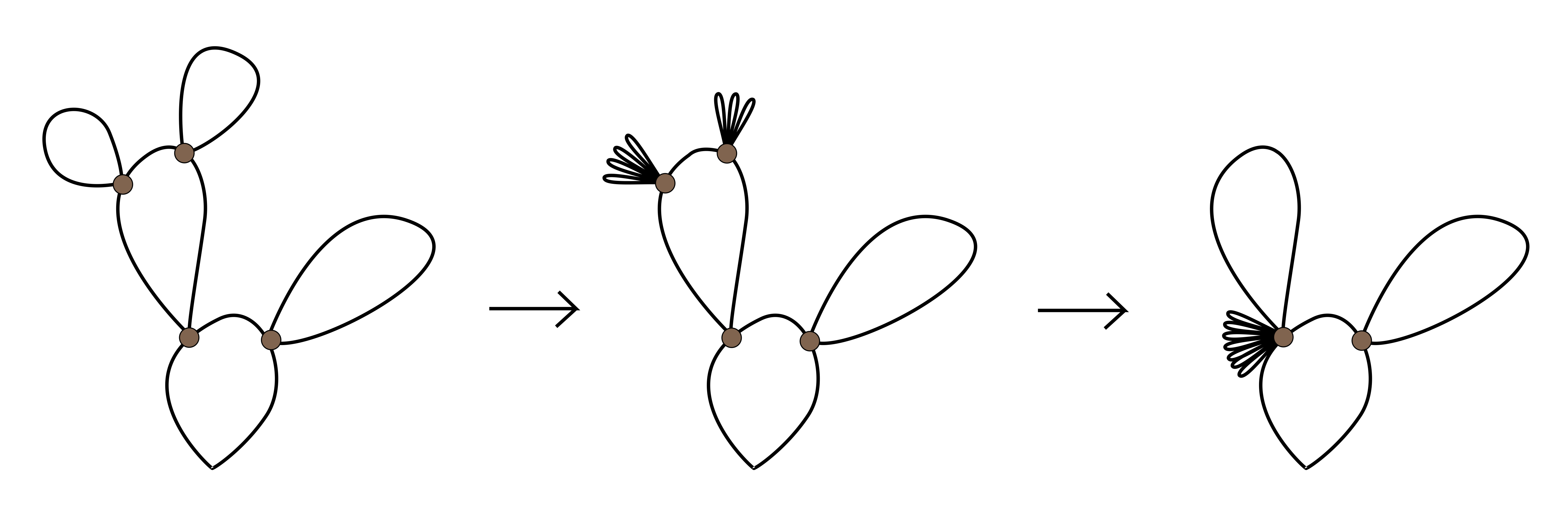}
	};
	
	\node[scale=1] at (-4.4,-2.25) {$n(C)=3$};
	
	\node[scale=1] at (0,-2.25) {$n(C')=3$};
	
	\node[scale=1] at (4.4,-2.25) {$n(C'')=2$};
	\end{tikzpicture}
	\end{center}
	\caption{Illustration of Steps 3 (left) and 4 (right) in the proof of Theorem~\ref{theorem.square.dismantable.to.log}.}\label{cactus-l}
\end{figure}

\begin{proof}
\begin{enumerate}
\item \textbf{Roadmap of the proof:}
Consider a cycle $c$ of length $n$ on $G$. Our goal is to prove that $c$ can be compressed into spines so that:
\[
d_G(c, t^{l(c)/2}) \le 2|V_G|(|\mathcal{C}_G^0|\alpha_G \log(n) + |\mathcal{C}_G^0| \text{diam}(G) + |V_G|)
\]
This ends the proof because, together with Lemma~\ref{lemma.walks.to.cycles} and Lemma~\ref{lemma.even}, this implies that $\gamma_G(n) = O(\log(n))$.

\item \textbf{Using the cactus representation to write $c$ relatively to its leaves:}
From Lemma~\ref{lemma.cactus} we know that there exists a cactus forest $(C_1,\ldots,C_l)$ such that $\pi(C_1,\ldots,C_l)=c$ with depth $k\le|V_G|$. Let us denote by $w^{(1)} , \ldots , w^{(m)}$ an enumeration of the highest level $\ell_k(C_1,\ldots,C_l)$ of this cactus forest such that there exist walks $\gamma^{(1)}, \dots , \gamma^{(m+1)}$ such that: 
\begin{equation}c=\gamma^{(1)} \odot w^{(1)} \odot \gamma^{(2)} \odot \ldots \odot \gamma^{(m)} \odot w^{(m)} \odot \gamma^{(m+1)}.\label{eq:leavesinc}
\end{equation}
Let us recall that each of the cycles $w^{(i)}$ is simple.

\item \textbf{Compressing the leaves into spines:} Let us consider $d^{(1)}, \dots , d^{(\tau)}$ an enumeration of $\mathcal{C}_G^0$. Let 
us construct a sequence of cycles $u^{(0)}, \dots. , u^{(\tau)}$ as follows. 
The aim is to recursively replace all $\xi(w^{(j)})$ in \eqref{eq:leavesinc} representing $d^{(i)}$ by spines.
First, $u^{(0)} = c$ and for all $p < \tau$, assuming 
that $u^{(0)}, \dots , u^{(p)}$ have been constructed, if there exist a sequence $\textbf{r}$, some $q \ge 1$, and some simple walks $v,w$ such that 
$u^{(p)} = v \odot \Gamma^{\textbf{r}}_{\textbf{z}} (d) \odot w$, where $d = \left(d^{(p+1)}\right)^q$ and $\textbf{z}$ consists in multiples of $l\left(d^{(p)}\right)$, 
then we set 
\[u^{(p+1)} := v \odot \Gamma^{\textbf{r}}_{\textbf{z}} ( t^{ql\left(d^{(p+1)}\right)/2}) \odot w,\]
Otherwise, $u^{(p+1)} :=u^{(p)}$.
Note that such representation exists if and only if $d^{(p+1)}$ was among $\xi(w^{(j)})$ in \eqref{eq:leavesinc}.

By Lemma~\ref{proposition.log.upper.bound.general} and Lemma~\ref{lemma:extension}, for all $p < \tau$,
\[d_G(u^{(p+1)}, u^{(p)}) \le \alpha_G \log(n) + \text{diam}(G).\]

As a result of this recursive elimination, $u:= u^{(\tau)}$ can be written as follows:
\[u := \gamma^{(1)} \odot \left(t^{(1)}\right)^{a_1} \odot \gamma^{(2)} \odot \left(t^{(2)}\right)^{a_2} \odot \cdots \odot \gamma^{(m_c)} \odot \left(t^{(m_c)}\right)^{a_{m_c}} \odot \gamma^{(m_c+1)},
\]
where each $t^{(i)}$ is a spine. 

Using Lemma~\ref{lem:comp-distances}, we have:
\[d_G(c,u) \le \tau \alpha_G \log(n) \le |\mathcal{C}_G^0| \left(\alpha_G \log(n)+ \text{diam}(G)\right).\]

Another way to describe $u$ is to say that we have $u = \pi(C'_1,\ldots,C'_l)$ where $(C'_1,\ldots,C'_l)$ is the cactus forest obtained from $(C_1,\ldots,C_l)$ by replacing each leaf $w^{(i)}$ in its $k$-th level by $a_i$ cacti $w^{(i,j)}$, $1 \le j \le a_i$ such that for all $i,j$, $\chi (w^{(i,j)}) = \chi(w^{(i)})$ and $\xi(w^{(i,j)}) = t^{(i)}$.
Figure~\ref{cactus-l} (left) illustrates the transformation of $(C_1,\ldots,C_l)$ into $(C'_1,\ldots,C'_l)$.

\item\textbf{Gathering the spines:} 
Now we bring every spine at depth $k$ to depth $k-1$. We consider two cases: 
\begin{enumerate}[(i)]
    \item \textbf{When $\boldsymbol{k>2}$.} For all $q \le l$, consider the cactus $C''_q$ obtained from $C'_q$ as follows. For every cactus $C^\ast$ in $\ell_{k-2}(C'_q)$ and every $m$, take all cacti $w^{(i,j)}$ that appears in $s(s(C')_m)$, remove them from $s(s(C')_m)$, and add them to $s(C')$ just before $s(C')_m$, keeping their relative order. Furthermore, fix $\chi(w^{(i,j)}) = \chi(s(C')_m)$.
    Let us denote $c' := \pi(C''_1,\ldots,C''_l)$.
    \item \textbf{When $\boldsymbol{k=2}$.} In this case,
    we consider $c' := \pi(C'')$, 
    where $C''$ is a cactus forest obtained from $(C'_1,\ldots,C'_l)$
    in the following way. For all $q \le l$, remove the cacti $w^{(i,j)}$ which appear in $s(C'_q)$ and insert them in the cactus forest in between $C'_q$ and $C'_{q+1}$.
\end{enumerate}

Using Lemma~\ref{lemma:movingbacktracks} on every such cactus $C^\ast$ and Remark~\ref{lemma.parallelisation.second}, we get that:
\[
d^\mathcal{R}_G (u, c') \le \max_{d \in \mathcal{C}_G^0} l(d) \le |V_G|\qquad\text{and thus}\qquad d_G (u, c') \le 2|V_G|.
\]
Since every leaf was moved 
one level down we have $n(C'') = n(C') - 1 = n(C_1, \ldots , C_l) - 1$.

Figure~\ref{cactus-l} (right) illustrates the transformation of $(C'_1,\ldots,C'_l)$ into $(C''_1,\ldots,C''_l)$.

\item \textbf{Iterating the process:} In steps 3 and 4, we found a cycle $c'$ represented by a cactus forest of depth $n(C_1,\ldots,C_l)-1$ such that $l(c)=l(c')$ and 
\[
d_G(c,c') \le |\mathcal{C}_G^0| \left(\alpha_G \log(n)+ \text{diam}(G) \right) + 2|V_G|.
\] 
By repeating this argument, and since $n(C_1,\ldots,C_l) \le |V_G|$, we find a cycle $f$ which is represented by a cactus forest of depth $1$ such that
\[
d_G(c,f) \le 2(|V_G| - 1) (|\mathcal{C}_G^0| \alpha_G \log(n) + |\mathcal{C}_G^0| \text{diam}(G) + |V_G|).
\]
By applying again Lemma~\ref{proposition.log.upper.bound.general} as in Step $3$ of the present proof, we have:
\[
d_G(f,t^{l(c)/2}) \le |\mathcal{C}_G^0|\alpha_G  \log(n),
\]
and thus, by triangular inequality:
\[
d_G(c, t^{l(c)/2}) \le |V_G|(|\mathcal{C}_G^0|\alpha_G \log(n) + |\mathcal{C}_G^0| \text{diam}(G) + |V_G|).
\]
This implies that the claim from the first step holds, and thus the theorem is proved.\qedhere
\end{enumerate}
\end{proof}

\section{A $\Theta(\log(n))$-phased block gluing Hom shift\label{section.log.existence}}

This section disproves R.~Pavlov and M.~Schraudner's conjecture~\cite{CM18} that this class is empty.

\begin{figure}[ht]
\begin{center}
\begin{tikzpicture}[scale=0.4]
\draw (-2,0) -- (0,0) -- (2,0);
\draw (-2,0) -- (-3,1.72) -- (-1,1.72);
\draw (2,0) -- (3,1.72) -- (1,1.72);
\draw (-1,1.72) -- (0,0) -- (1,1.72);
\draw (-1,-1.72) -- (0,0) -- (1,-1.72);
\draw (-1,-1.72) -- (0,-3.44) -- (1,-1.72);
\draw (-1,1.72) -- (0,3.44) -- (1,1.72);
\draw (-3,1.72) -- (0,5.88) -- (0,3.44);
\node at (0,6.7) {$\epsilon_1$};
\draw[fill=gray!90] (2,0) circle (3pt);
\node at (1.5,0.5) {$\mu_6$};
\draw[fill=gray!90] (-2,0) circle (3pt);
\node at (-1.5,0.5) {$\mu_3$};
\draw (0,5.88) -- (3,1.72);
\draw (-2,0) -- (-3,-1.72) -- (-1,-1.72);
\draw (2,0) -- (3,-1.72) -- (1,-1.72);
\draw[fill=gray!90] (1,-1.72) circle (3pt);
\node at (1.75,-2.5) {$\mu_5$};
\draw[fill=gray!90] (-1,-1.72) circle (3pt);
\node at (-1.75,-2.5) {$\mu_4$};
\draw[fill=gray!90] (3,-1.72) circle (3pt);
\node at (3.5,-1) {$\delta_3$};
\draw[fill=gray!90] (-3,-1.72) circle (3pt);
\node at (-3.5,-1) {$\delta_2$};

\draw[fill=gray!90] (1,1.72) circle (3pt);
\node at (1.5,2.25) {$\mu_1$};
\draw[fill=gray!90] (-1,1.72) circle (3pt);
\node at (-1.5,2.25) {$\mu_2$};
\draw[fill=gray!90] (3,1.72) circle (3pt);
\node at (3.65,2) {$\gamma_3$};
\draw[fill=gray!90] (-3,1.72) circle (3pt);
\node at (-3.65,2) {$\gamma_1$};

\node at (6,-3.75) {$\epsilon_3$};
\node at (-6,-3.75) {$\epsilon_2$};
\draw[fill=gray!90] (-5.225,-2.945) circle (3pt);
\draw[fill=gray!90] (5.225,-2.945) circle (3pt);
\draw (3,-1.72) -- (5.225,-2.945) -- (0,-3.44);
\node at (0,-4.6) {$\gamma_2$};
\draw[fill=gray!90] (0,-3.44) circle (3pt);
\draw[fill=gray!90] (0,5.88) circle (3pt);
\draw[fill=gray!90] (0,0) circle (3pt);
\node at (-1,-0.5) {$\omega$};
\draw[fill=gray!90] (0,3.44) circle (3pt);
\node at (-0.75,3.7) {$\delta_1$};
\draw (5.225,-2.945) -- (3,1.72);

\draw (-3,-1.72) -- (-5.225,-2.945) -- (0,-3.44);
\draw (-5.225,-2.945) -- (-3,1.72);
\end{tikzpicture}
\end{center}
\caption{The Ken-katabami graph.\label{figure.kenkatabami}}
\end{figure}
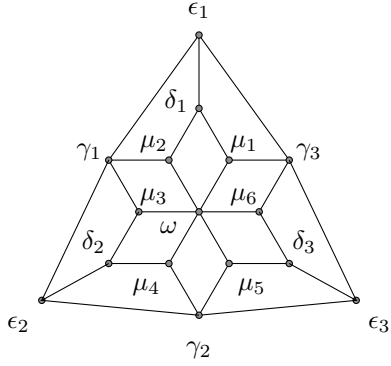
\begin{theorem}\label{theorem.kenkatabami}
There exists a graph $K$ such that $X_K$ is $\Theta(\log(n))$-phased block gluing.
\end{theorem}

\begin{proof}
$K$ is the Ken-katabami graph shown on Figure~\ref{figure.kenkatabami} (we
use the notations of the figure for the vertices)\footnote{The name comes from a visual similarity with the Ken-katabami (\begin{CJK}{UTF8}{ipxm}剣片喰\end{CJK}) Japanese crest, see \url{https://commons.wikimedia.org/wiki/File:Ken-Katabami_inverted.png}}. Since the graph is clearly square decomposable, Theorem~\ref{theorem.square.dismantable.to.log} applies, and $\gamma_K(n) = O(\log(n))$. Thus it is sufficient to prove that $\log(n) = O(\gamma_K(n))$.

\begin{enumerate}
\item \textbf{Notations:}
Let us denote by $c = \epsilon_1 \gamma_1 \epsilon_2 \gamma_2 \epsilon_3 \gamma_3 \epsilon_1$ the anti-clockwise exterior cycle of the graph $K$.

For any walk $p$ on the graph $K$, a $c$-block of $p$ is a maximal word of the form $c^n$ in $p$. This $n$ is called the order of this $c$-block. We also denote by $\mu_c(p)$ the maximal order of a $c$-block.

\item \textbf{Lower bound on $\mu_c (p_2)$ for $p_2$ neighbor of 
$p_1$:}

\textbf{(i) Claim:} We claim that for $n$ large enough and two walks $p, q$ that are neighbours in $\Delta_K^{N}$, the following formula holds: 
\[\mu_c(p) \ge\frac{1}{2} \mu_c(q) -3.\]

\textbf{(ii) Characteristics of the neighbors of $c^n$ in 
$\Delta_K^{6n}$:} For all $n \ge4$, let us prove that a neighbor $p$ of $q = c^n$ in the graph $\Delta_K^{6n}$ is of the form 
\begin{enumerate}
\item $u\odot w\odot v$ with $|w| \le 3$, 
\item $u\odot w$ or $w\odot v$ with $|w|=1$, 
\item $u$ or $v$,
\end{enumerate}
where the words $u$ and $v$ are respectively a right shift of a prefix of $c^n$ and a left shift of a suffix of $c^n$.

The only common neighbor of $\epsilon_i$ and $\epsilon_{i+1}$ is $\gamma_i$ and the only common neighbor of $\gamma_i$ and $\gamma_{i+1}$ is 
$\epsilon_i$ (where for technical reasons the index $i$ is identified with the corresponding element of $\mathbb{Z}/3\mathbb{Z}$). Take an arbitrary $k$, and assume for simplicity that $q_k= \epsilon_2$ (other cases are processed similarly). Then possible values for $p_k$ are $\gamma_1, \gamma_2$ or $\delta_2$, and we can check that the only possibilities for $p$ are the following:

\begin{figure}[h!]
\begin{tabular}{ccc}
\begin{tikzpicture}[scale=0.65]
\draw[step=1,black,thin] (0,0) grid (5,2);
\node at (2.5,3.3) {$k+$};

\node at (0.5,2.5) {$-2$};
\node at (1.5,2.5) {$-1$};
\node at (2.5,2.5) {$0$};
\node at (3.5,2.5) {$+1$};
\node at (4.5,2.5) {$+2$};

\node at (-0.5, 1.5) {$q$};
\node at (-0.5, 0.5) {$p$};

\node at (.5, 1.5) {$\epsilon_1$};
\node at (1.5, 1.5) {$\gamma_1$};
\node at (2.5, 1.5) {$\epsilon_2$};
\node at (3.5, 1.5) {$\gamma_2$};
\node at (4.5, 1.5) {$\epsilon_3$};

\node at (2.5, 0.5) {$\gamma_1$};
\node at (3.5, 0.5) {$\epsilon_2$};
\node at (4.5, 0.5) {$\gamma_2$};
\end{tikzpicture}
&
\begin{tikzpicture}[scale=0.65]
\draw[step=1,black,thin] (0,0) grid (5,2);
\node at (2.5,3.3) {$k+$};

\node at (0.5,2.5) {$-2$};
\node at (1.5,2.5) {$-1$};
\node at (2.5,2.5) {$0$};
\node at (3.5,2.5) {$+1$};
\node at (4.5,2.5) {$+2$};

\node at (-0.5, 1.5) {$q$};
\node at (-0.5, 0.5) {$p$};

\node at (.5, 1.5) {$\epsilon_1$};
\node at (1.5, 1.5) {$\gamma_1$};
\node at (2.5, 1.5) {$\epsilon_2$};
\node at (3.5, 1.5) {$\gamma_2$};
\node at (4.5, 1.5) {$\epsilon_3$};
\node at (.5, 0.5) {$\gamma_1$};
\node at (1.5, 0.5) {$\epsilon_2$};
\node at (2.5, 0.5) {$\gamma_2$};
\end{tikzpicture}
&
\begin{tikzpicture}[scale=0.65]
\draw[step=1,black,thin] (0,0) grid (5,2);
\node at (2.5,3.3) {$k+$};

\node at (0.5,2.5) {$-2$};
\node at (1.5,2.5) {$-1$};
\node at (2.5,2.5) {$0$};
\node at (3.5,2.5) {$+1$};
\node at (4.5,2.5) {$+2$};

\node at (-0.5, 1.5) {$q$};
\node at (-0.5, 0.5) {$p$};

\node at (.5, 1.5) {$\epsilon_1$};
\node at (1.5, 1.5) {$\gamma_1$};
\node at (2.5, 1.5) {$\epsilon_2$};
\node at (3.5, 1.5) {$\gamma_2$};
\node at (4.5, 1.5) {$\epsilon_3$};

\node at (.5, 0.5) {$\gamma_1$};
\node at (1.5, 0.5) {$\nicefrac{\epsilon_2}{\mu_3}$};
\node at (2.5, 0.5) {$\delta_2$};
\node at (3.5, 0.5) {$\nicefrac{\epsilon_2}{\mu_4}$};
\node at (4.5, 0.5) {$\gamma_2$};
\end{tikzpicture}
\end{tabular}
\end{figure}

In the leftmost case ($p_k= q_{k-1}$), we see that $p_{\llbracket k, 6n\rrbracket}= q_{\llbracket k,6n\rrbracket-1} $.
Similarly, in the center case, ($p_k= q_{k+1}$), we have $p_{\llbracket 0, k\rrbracket}= q_{\llbracket 0, k\rrbracket+1}$.
Finally, in the rightmost case ($p_k\notin\{q_{k-1}, q_{k+1}\}$), we can see that $p_{\llbracket 0, k-2\rrbracket}= q_{\llbracket 0, k-2\rrbracket +1}$ or $p_{\llbracket k+2, 6n\rrbracket}= q_{\llbracket k+2,6n\rrbracket-1}$.

Therefore:
\begin{itemize}
\item if for any $k$, $p_k\notin\{q_{k-1}, q_{k+1}\}$, we are in case (a).
\item if $p_k= q_{k+1}$ for all $k$, we are in case (b) or (c); similarly if $p_k= q_{k-1}$.
\item we cannot have $p_k = q_{k-1}$ and $p_{k+1} = q_{k+2}$ since $q_{k-1}$ and $q_{k+2}$ are not neighbours in the graph.
\end{itemize}

\textbf{(iii) Lower bound on $\mu_c(p)$ for $p$ a neighbor of $c^n$ in $\Delta_K^{6n}$:}

If $u$ is a prefix or suffix of $c^n$, it is clear that $\mu_c(u) \ge\lfloor \frac{l(u)}{6} \rfloor$. Furthermore, for any walk $p'$, $\mu_c(p') \ge \mu_c(p'')-1$ for any shift $p''$ of $p'$.

As a consequence of point \textbf{(ii)}, any neighbor $p$ of $c^n$ in $\Delta_K^{6n}$ has a subword which is a shift of a prefix or suffix of $c^n$ and whose length is at least $\frac{6n-3}{2}$. Thus we have that 
\[\mu_c(p) \ge \left\lfloor\frac{1}{6} \frac{6n-3}{2} - 1\right\rfloor \ge\frac{n}{2} - 3.\]

\textbf{(iv) Proof of the claim:} Let us consider $p$ and $q$ two words which are neighbors in $\Delta_K^{N}$. By definition, if $n=\mu_c(p)$, then 
$p$ has $c^n$ as a subword. The corresponding subword of $q$ is a neighbour of $c^n$ in $\Delta_K^{6n}$, which means that $\mu_c (q) \ge \frac{n}{2} - 3 = \frac{\mu_c(p)}{2}-3$.

\item \textbf{Lower bound on $\gamma_K(n)$:} For any integer $n$, consider $c^n$ and $q_n$ any walk of length $6n$ that does not contain $c$, such as a repeated spine. Thus $\mu_c(c^n) = n$ and $\mu_c(q_n) = 0$.
By the previous claim, any walk $p_0, \dots, p_m$ in $\Delta_K^{6n}$ satisfies $\mu_c (p_{i+1}) \ge \frac{\mu_c(p_i)}{2}-3$ for all $i$.
It follows that $\log(n) = O(d_G(c^n, q_n))$. By Lemma~\ref{lemma.even} and Proposition~\ref{prop:block-gluing}, this implies that $\log(n) = O(\gamma_K(n))$.\qedhere


\end{enumerate}
\end{proof}

\section{Open problems\label{section.conclusion}} 

We leave two main problems for further research. The \textit{first problem} is the classification of block gluing classes for Hom shifts. We conjecture the following: 

\begin{conjecture}
The only possible classes of gap functions for Hom shifts are $\Theta(1)$, $\Theta(\log(n))$ and $\Theta(n)$.
\end{conjecture}

Any tree provides a $\Theta(1)$-block gluing Hom shift, and there are several nontrivial examples of $\Theta(1)$-phased block gluing Hom shifts (see Figure~\ref{figure.lonely.cluster}).
We can construct more examples of $\Theta(\log(n))$-block gluing Hom shifts by 'gluing' these graphs together (see Figure~\ref{figure.counterexample}). What intuitively separates these two sets of graphs is whether every cycle can be deformed to a trivial cycle so that no intermediate cycle is larger that the original. However formalizing this intuition has proven difficult.

\begin{figure}[ht]
\begin{center}
\begin{tikzpicture}[scale=0.4]
\draw[fill=gray!90] (0,0) circle (3pt);
\draw[fill=gray!90] (2,2) circle (3pt);
\draw[fill=gray!90] (2,0) circle (3pt);
\draw[fill=gray!90] (0,2) circle (3pt);
\draw[fill=gray!90] (3,3) circle (3pt);
\draw[fill=gray!90] (3,1) circle (3pt);
\draw[fill=gray!90] (1,3) circle (3pt);
\draw[fill=gray!90] (4,4) circle (3pt);
\draw[fill=gray!90] (4,1) circle (3pt);
\draw[fill=gray!90] (2,-1) circle (3pt);
\draw[fill=gray!90] (-1,-1) circle (3pt);
\draw[fill=gray!90] (-1,2) circle (3pt);
\draw[fill=gray!90] (1,4) circle (3pt);
\draw (0,0) -- (0,2) -- (2,2) -- (2,0) -- (0,0);
\draw (0,2) -- (1,3) -- (3,3) -- (2,2);
\draw (3,3) -- (3,1) -- (2,0);
\draw (3,3) -- (4,4) -- (4,1) -- (3,1);
\draw (2,0) -- (2,-1) -- (4,1);
\draw (2,-1) -- (-1,-1) -- (0,0);
\draw (-1,-1) -- (-1,2) -- (0,2);
\draw (-1,2) -- (1,4) -- (1,3);
\draw (1,4) -- (4,4);

\begin{scope}[xshift=11cm,yshift=1.5cm]
\draw[fill=gray!90] (0,0) circle (3pt);
\draw[fill=gray!90] (1,1) circle (3pt);
\draw[fill=gray!90] (2,0) circle (3pt);
\draw[fill=gray!90] (1,-1) circle (3pt);
\draw[fill=gray!90] (1,3) circle (3pt);
\draw[fill=gray!90] (1,-3) circle (3pt);
\draw[fill=gray!90] (3,1) circle (3pt);
\draw[fill=gray!90] (-2,0) circle (3pt);
\draw[fill=gray!90] (3,-1) circle (3pt);
\draw[fill=gray!90] (4,0) circle (3pt);
\draw (0,0) -- (1,1) -- (2,0) -- (1,-1) -- (0,0);
\draw (1,1) -- (1,3) -- (3,1);
\draw (1,3) -- (-2,0) -- (0,0);
\draw (-2,0) -- (1,-3) -- (1,-1);
\draw (1,-3) -- (3,-1);
\draw (2,0) -- (3,1) -- (4,0) -- (3,-1) -- (2,0);
\end{scope}
\end{tikzpicture}
\end{center}
\caption{Example of graphs whose Hom shift is $\Theta(1)$-phased block gluing.\label{figure.lonely.cluster}}
\end{figure}
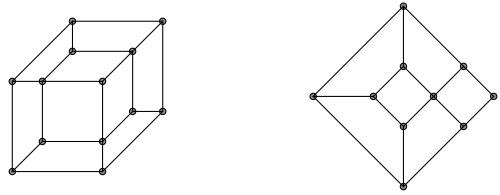

\begin{figure}[ht]
\begin{center}
\begin{tikzpicture}[scale=0.4]
\draw (0,0) -- (0,2) -- (2,2) -- (2,0) -- (0,0);
\draw (0,2) -- (1,3) -- (3,3) -- (2,2);
\draw (3,3) -- (3,1) -- (2,0);
\draw (3,3) -- (4,4) -- (4,1) -- (3,1);
\draw (2,0) -- (2,-1) -- (4,1);
\draw (2,-1) -- (-1,-1) -- (0,0);
\draw (-1,-1) -- (-1,2) -- (0,2);
\draw (-1,2) -- (1,4) -- (1,3);
\draw (1,4) -- (4,4);
\draw[fill=gray!90] (0,0) circle (3pt);
\draw[fill=gray!90] (2,2) circle (3pt);
\draw[fill=gray!90] (2,0) circle (3pt);
\draw[fill=gray!90] (0,2) circle (3pt);
\draw[fill=gray!90] (3,3) circle (3pt);
\draw[fill=gray!90] (3,1) circle (3pt);
\draw[fill=gray!90] (1,3) circle (3pt);
\draw[fill=gray!90] (4,4) circle (3pt);
\draw[fill=gray!90] (4,1) circle (3pt);
\draw[fill=gray!90] (2,-1) circle (3pt);
\draw[fill=gray!90] (-1,-1) circle (3pt);
\draw[fill=gray!90] (-1,2) circle (3pt);
\draw[fill=gray!90] (1,4) circle (3pt);

\begin{scope}[xshift=-5cm,yshift=-5cm]
\draw (0,0) -- (0,2) -- (2,2) -- (2,0) -- (0,0);
\draw (0,2) -- (1,3) -- (3,3) -- (2,2);
\draw (3,3) -- (3,1) -- (2,0);
\draw (3,3) -- (4,4) -- (4,1) -- (3,1);
\draw (2,0) -- (2,-1) -- (4,1);
\draw (2,-1) -- (-1,-1) -- (0,0);
\draw (-1,-1) -- (-1,2) -- (0,2);
\draw (-1,2) -- (1,4) -- (1,3);
\draw (1,4) -- (4,4);
\draw[fill=gray!90] (0,0) circle (3pt);
\draw[fill=gray!90] (2,2) circle (3pt);
\draw[fill=gray!90] (2,0) circle (3pt);
\draw[fill=gray!90] (0,2) circle (3pt);
\draw[fill=gray!90] (3,3) circle (3pt);
\draw[fill=gray!90] (3,1) circle (3pt);
\draw[fill=gray!90] (1,3) circle (3pt);
\draw[fill=gray!90] (4,4) circle (3pt);
\draw[fill=gray!90] (4,1) circle (3pt);
\draw[fill=gray!90] (2,-1) circle (3pt);
\draw[fill=gray!90] (-1,-1) circle (3pt);
\draw[fill=gray!90] (-1,2) circle (3pt);
\draw[fill=gray!90] (1,4) circle (3pt);
\end{scope}
\end{tikzpicture}
\qquad
\begin{tikzpicture}[scale=0.4]
\draw (0,0) -- (1,1) -- (2,0) -- (1,-1) -- (0,0);
\draw (1,1) -- (1,3) -- (3,1);
\draw (1,3) -- (-2,0) -- (0,0);
\draw (-2,0) -- (1,-3) -- (1,-1);
\draw (1,-3) -- (3,-1);
\draw (2,0) -- (3,1) -- (4,0) -- (3,-1) -- (2,0);

\draw (4,0) -- (5,1) -- (6,0) -- (5,-1) -- (4,0);
\draw (5,1) -- (7,3) -- (10,0) -- (8,0);
\draw (7,1) -- (7,3);
\draw (5,-1) -- (7,-3) -- (10,0);
\draw (7,-3) -- (7,-1);
\draw (6,0) -- (7,1) -- (8,0) -- (7,-1) -- (6,0);

\draw[fill=gray!90] (0,0) circle (3pt);
\draw[fill=gray!90] (1,1) circle (3pt);
\draw[fill=gray!90] (2,0) circle (3pt);
\draw[fill=gray!90] (1,-1) circle (3pt);
\draw[fill=gray!90] (1,3) circle (3pt);
\draw[fill=gray!90] (1,-3) circle (3pt);
\draw[fill=gray!90] (3,1) circle (3pt);
\draw[fill=gray!90] (-2,0) circle (3pt);
\draw[fill=gray!90] (3,-1) circle (3pt);
\draw[fill=gray!90] (4,0) circle (3pt);

\draw[fill=gray!90] (5,1) circle (3pt);
\draw[fill=gray!90] (6,0) circle (3pt);
\draw[fill=gray!90] (5,-1) circle (3pt);
\draw[fill=gray!90] (7,1) circle (3pt);
\draw[fill=gray!90] (7,-1) circle (3pt);
\draw[fill=gray!90] (7,3) circle (3pt);
\draw[fill=gray!90] (7,-3) circle (3pt);
\draw[fill=gray!90] (8,0) circle (3pt);
\draw[fill=gray!90] (10,0) circle (3pt);
\end{tikzpicture}

\end{center}
\caption{Example of graphs 
whose Hom shift is $\Theta(\log(n))$-phased block gluing.
\label{figure.counterexample}}
\end{figure}
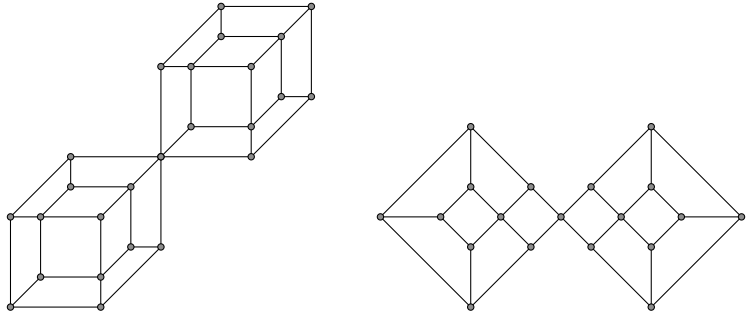

The \textit{second problem} is finding an algorithm which, 
provided a graph $G$, decides in which class (for $\Theta$) its gap function belongs to. In particular, since we know 
that there is no intermediate class between $O(\log(n))$ and $\Theta(n)$ and that the class to which a graph belongs depends on whether or not its square cover is finite, the key question seems to be:

\begin{question}
Is there an algorithm which decides, given a finite graph $G$, whether $\quadcover{G}$ is finite or not ?
\end{question}

This question seems to be close to known undecidability results: the square cover is finite if and only if the quotient of the \textit{fundamental group} of $G$ by squares of $G$ is finite, and it is not possible in general to decide if a group defined by a finite number of generators and relations is finite or not. This is a consequence of the Adian-Rabin theorem \cite{A57, R58} -- see \cite{NB22} for a translation and exposition of Adian's work.

Additional open questions on Hom shifts can be found in \cite{C17, CM18}. For example, we do not know how to characterise mixing properties relative to general shapes (not only rectangular) in Hom shifts.

\section*{Acknowledgements}

The second author wishes to thank Nishant Chandgotia for introducing him to Hom shifts and many fruitful discussions, and Jan van der Heuvel along with the organizers of CoRe 2019 where the Ken-Katabami graph was found. (\textit{As well the first author is thankful to the second author for introducing him to Hom shifts, 
as well as the third author is thankful to the first one}). 

The second author was partially supported by a LISN internal project. The research of the first and the third author was supported in part by  National Science Centre, Poland (NCN), grant no. 2019/35/B/ST1/02239. 

 The authors declare they have no financial interests.

\bibliographystyle{alpha}
\bibliography{biblio.bib}
\end{document}